\documentclass[reqno,11pt]{amsart}
\usepackage[colorlinks=true, linkcolor=blue, citecolor=blue]{hyperref}

\usepackage{amssymb}
\usepackage{amsmath, graphicx, rotating}
\usepackage{color}
\usepackage{soul}
\usepackage[dvipsnames]{xcolor}

\usepackage{ifthen}
\usepackage{xkeyval}
\usepackage{todonotes}
\usepackage[T1]{fontenc}
\usepackage{lmodern}
\usepackage[english]{babel}

\usepackage{ upgreek }
\usepackage{stmaryrd}
\SetSymbolFont{stmry}{bold}{U}{stmry}{m}{n}
\usepackage{amsthm}
\usepackage{float}

\usepackage{ bbm }
\usepackage{ stmaryrd }
\usepackage{ mathrsfs }
\usepackage{ frcursive }
\usepackage{ comment }

\usepackage{pgf, tikz}
\usetikzlibrary{shapes}
\usepackage{varioref}
\usepackage{enumitem}
\usepackage{ulem}

\definecolor{rouge}{rgb}{0.7,0.00,0.00}
\definecolor{vert}{rgb}{0.00,0.5,0.00}
\definecolor{bleu}{rgb}{0.00,0.00,0.8}
\usepackage[margin=1in]{geometry}
\newtheorem{theorem}{Theorem}[section]
\newtheorem*{theorem*}{Theorem}
\newtheorem{lemma}[theorem]{Lemma}
\newtheorem{definition}[theorem]{Definition}
\newtheorem{corollary}[theorem]{Corollary}
\newtheorem{proposition}[theorem]{Proposition}

\labelformat{hypothesis}{\textbf{M\kern-0.1mm#1}}

\newtheorem{condition}{Condition}

\newtheorem{conditionA}{A\kern-0.1mm}
\labelformat{conditionA}{\textbf{A\kern-0.1mm#1}}

\newtheorem{conditionB}{B\kern-0.1mm}
\labelformat{conditionB}{\textbf{B\kern-0.1mm#1}}

\newtheorem{conditionC}{C\kern-0.1mm}
\labelformat{conditionC}{\textbf{C\kern-0.1mm#1}}

\theoremstyle{definition}
\newtheorem{example}[theorem]{Example}
\newtheorem{remark}[theorem]{Remark}

\def \eref#1{\hbox{(\ref{#1})}}

\numberwithin{equation}{section}

\def\geq{\geqslant}
\def\leq{\leqslant}

\def\RR{\mathbb{R}}
\def\PP{\mathbb{P}}
\def\EE{\mathbb{E}}
\def\NN{\mathbb{N}}

\def\cM{{\mathcal M}}

\def\sI{{\mathscr I}}
\def\sJ{{\mathscr J}}

\def\cP{{\mathcal  P}}

\def\cL{{\mathcal L}}

\def\vare{{\varepsilon}}
\def \eref#1{\hbox{(\ref{#1})}}

\def\EE{\mathbb{ E}}

\allowdisplaybreaks

\begin{document}

\title[Strong averaging principle for nonautonomous multi-scale SPDEs]
{Strong averaging principle for nonautonomous multi-scale SPDEs with fully local monotone and almost periodic coefficients}

\author{Mengyu Cheng}
\address{Cheng M.: School of Mathematical Sciences, Dalian University of Technology, Dalian, 116024, People's Republic of China}
\email{mycheng@dlut.edu.cn}

\author{Xiaobin Sun}
\address{Sun, X.: School of Mathematics and Statistics/RIMS, Jiangsu Normal University, Xuzhou, 221116, People's Republic of China}
\email{xbsun@jsnu.edu.cn}

\author{Yingchao Xie}
\address{Xie, Y.: School of Mathematics and Statistics/RIMS,  Jiangsu Normal University, Xuzhou, 221116, People's Republic of China}
\email{ycxie@jsnu.edu.cn}

\begin{abstract}
In this paper, we consider a class of nonautonomous multi-scale stochastic partial differential equations with fully local monotone coefficients. By introducing the evolution system of measures for time-inhomogeneous Markov semigroups, we study the averaging principle for such kind of system. Specifically, we first prove the slow component in the multi-scale stochastic system converges strongly to the solution of an averaged equation, whose coefficients retain the dependence of the scaling parameter.
Furthermore, if the coefficients satisfy uniformly almost periodic conditions, we establish that the slow component converges strongly to the solution of another averaged equation, whose coefficients are independent of the scaling parameter.
The main contribution of this paper extends the basic nonautonomous framework investigated by Cheng and Liu in \cite{CL2023} to a fully coupled framework, as well as the autonomous framework explored by Liu et al. in \cite{LRSX2023} to the more general nonautonomous framework. Additionally, we improve the locally monotone coefficients discussed in \cite{CL2023,LRSX2023} to the fully local monotone coefficients, thus our results can be applied to a wide range of cases in nonlinear nonautonomous stochastic partial differential equations,
such as multi-scale stochastic Cahn-Hilliard-heat equation and
multi-scale stochastic liquid-crystal-porous-media equation.
\end{abstract}
\date{\today}

\maketitle
\tableofcontents
\section{Introduction}

\subsection{Background and motivation}
Multi-scale systems are widely used in various fields like nonlinear oscillations, chemical kinetics, biology, and climate dynamics, see e.g. \cite{BR2017,EE2003} and the references therein. Typically, multi-scale stochastic differential equations (SDEs) can be represented by two interconnected SDEs that correspond to the "slow" and "fast" components, as established in the foundational work by Khasminskii in \cite{K1968}. The averaging principle plays a crucial role in characterizing the asymptotic behavior of the slow components, indicating that the slow component will converge to the so-called averaged equation, where coefficients are derived by averaging out the fast component, thereby making it significantly simpler than the original equation.

Since the pioneering works \cite{C2009, CF2009} by Cerrai and Freidlin in 2009, which first established the averaging principle for stochastic reaction-diffusion equations,
there has been a growing interest among researchers in the averaging principle for multi-scale stochastic partial differential equations (SPDEs) over the last fifteen years, see e.g. \cite{B2012,C2011,De2022,DSXZ2018,FWL2015,G2023,GS2024,HLL2022,PXY2017,RXY2023,SG2022,SX2023,WR2012,WY2022}. Specifically, Cerrai in \cite{C2009} considered a class of multi-scale SPDEs with Dirichlet boundary conditions in a bounded domain $D \subset \mathbb{R}^{d} (d > 1)$:
\begin{eqnarray*}\label{1.1}
\left\{
\begin{array}{l}
\displaystyle
d u_{\varepsilon}(t) = \left[A_1 u_{\varepsilon}(t) + B_1(u_{\varepsilon}(t), v_{\varepsilon}(t))\right]dt + G_1(u_{\varepsilon}(t), v_{\varepsilon}(t)) dW^{Q_1}_t, \vspace{2mm}\\
\displaystyle
d v_{\varepsilon}(t) = \frac{1}{\varepsilon}\left[A_2 v_{\varepsilon}(t) + B_2(u_{\varepsilon}(t), v_{\varepsilon}(t))\right]dt + \frac{1}{\sqrt{\varepsilon}}G_2(u_{\varepsilon}(t), v_{\varepsilon}(t)) dW^{Q_2}_t, \vspace{2mm}\\
u_{\varepsilon}(0) = x \in L^{2}(D), v_{\varepsilon}(0) = y \in L^{2}(D),
\end{array}
\right.
\end{eqnarray*}
where $W^{Q_1}_t$ and $W^{Q_2}_t$ are Gaussian noises with covariance operators $Q_1$ and $Q_2$, respectively. The operators $A_1$ and $A_2$ are second-order uniformly elliptic operators. The coefficients $B_1, B_2: L^{2}(D) \times L^{2}(D) \rightarrow L^{2}(D)$, and $G_1, G_2: L^{2}(D) \times L^{2}(D) \rightarrow \mathcal{L}(L^{2}(D); L^1(D)) \cap \mathcal{L}(L^{\infty}(D); L^2(D))$ are Lipschitz continuous. Furthermore, assume the existence of two Lipschitz continuous mappings $\bar{B}_1: L^2(D) \rightarrow L^2(D)$ and $\bar{G}_1: L^2(D) \rightarrow \mathcal{L}(L^{\infty}(D); L^2(D))$ such that for any $T > 0$, $t \geq 0$, and $x, y \in L^2(D)$,
\begin{align*}
&\mathbb{E}\left|\frac{1}{T}\int_{t}^{t+T} \langle B_1(x, v^{x,y}(s)), h \rangle_{L^{2}(D)} ds - \langle \bar{B}_1(x), h \rangle_{L^{2}(D)}\right| \\
\leq&\alpha(T)(1 + |x|_{L^{2}(D)} + |y|_{L^{2}(D)}).
\end{align*}
For any $h \in L^2(D)$ and $h, k \in L^{\infty}(D)$ satisfying
\begin{align*}
&\left|\frac{1}{T}\int_{t}^{t+T}\mathbb{E}\langle G_1(x, v^{x,y}(s))h, G_1(x, v^{x,y}(s))k \rangle_{L^{2}(D)} ds - \langle \bar{G}_1(x)h, \bar{G}_1(x)k \rangle_{L^{2}(D)}\right| \\
\leq&\alpha(T)(1 + |x|^2 + |y|) |h|_{L^{\infty}(D)} |k|_{L^{\infty}(D)},
\end{align*}
where $\alpha(T) \rightarrow 0$ as $T \rightarrow \infty$. Here, $v^{x,y}(t)$ is the unique solution to the following frozen equation:
\begin{align}
d v^{x,y}(t) = A_2 v^{x,y}(t) dt + B_2(x, v^{x,y}(t)) dt + G_2(x, v^{x,y}(t)) dW^{Q_2}_t, \quad v^{x,y}(0) = y.\label{IfrozenE}
\end{align}
Cerrai \cite{C2009} proved that the slow component $u_{\varepsilon}$ converges in distribution to $\bar{u}$, i.e.,
$$
\mathcal{L}(u^{\varepsilon}) \rightarrow \mathcal{L}(\bar{u}) \quad \text{in} \quad C([0,T]; L^{2}(D)),
$$
where $\mathcal{L}(u^{\varepsilon})$ denotes the distribution of $u^{\varepsilon}$, and $\bar{u}$ is the unique solution to the following averaged equation:
$$
d \bar{u}(t) = A_1 \bar{u}(t) dt + \bar{B}_1(\bar{u}(t)) dt + \bar{G}_1(\bar{u}(t)) dW^{Q_1}_t.
$$
Moreover, in the case the diffusion coefficient $G_1$ in the slow equation does not depend on the fast oscillating variable $v_{\varepsilon}$, it is also proved that the convergence of $u^{\varepsilon}$ to $\bar{u}$ is in probability.

Certainly, a key step of proving the averaging principle is understanding the related averaging equation, with a primary focus on how to precisely define the averaging coefficient. Based on the general theory of ergodicity, the averaging coefficient is influenced by the invariant measure of the frozen equation. More precisely, if the frozen equation \eref{IfrozenE} admits a unique invariant measure $\mu^{x}$, then it follows
$$
\bar{B}_1(x)=\int_{L^2(D)}B_1(x,y)\mu^{x}(dy),\quad \bar{G_1}(x)=\sqrt{S(x)},
$$
where $S(x): L^2(D)\rightarrow \mathcal{L}(L^{2}(D))$ satisfies
$$
\langle S(x)h, k\rangle = \int_{L^2(D)}\langle G_1(x, z)h, G_1(x, z)k\rangle \mu^x(dz),\quad  h,k\in L^2(D).
$$

The majority of framework concentrate on the case where the frozen equation is an autonomous (time-independent) SPDEs. This focus is due to the extensive theoretical work on the existence and uniqueness of invariant measures. However, nonautonomous (time-dependent) frozen equation exhibits richer dynamic characteristics, leading to the understanding of the so-called equilibrium point becomes extremely complex, that is, considering the following frozen equation:
\begin{align*}
d v^{x,y}(t) = A_2(t) v^{x,y}(t) dt + B_2(t,x, v^{x,y}(t)) dt + G_2(t,x, v^{x,y}(t)) dW^{Q_2}_t.
\end{align*}
Then the corresponding multi-scale SPDEs can be expressed in the following manner:
\begin{eqnarray}\label{1.2}
\left\{
\begin{array}{l}
\displaystyle
d u_{\varepsilon}(t) = \left[A_1 u_{\varepsilon}(t) + B_1(u_{\varepsilon}(t), v_{\varepsilon}(t))\right]dt + G_1(u_{\varepsilon}(t), v_{\varepsilon}(t)) dW^{Q_1}_t, \vspace{2mm}\\
\displaystyle
d v_{\varepsilon}(t) = \frac{1}{\varepsilon}\left[A_2(t/\varepsilon) v_{\varepsilon}(t) + B_2(t/\varepsilon,u_{\varepsilon}(t), v_{\varepsilon}(t))\right]dt + \frac{1}{\sqrt{\varepsilon}}G_2(t/\varepsilon,u_{\varepsilon}(t), v_{\varepsilon}(t)) dW^{Q_2}_t, \vspace{2mm}\\
u_{\varepsilon}(0) = x \in L^{2}(D), v_{\varepsilon}(0) = y \in L^{2}(D).
\end{array}
\right.
\end{eqnarray}
It is important to mention that the system described above is nonautonomous.

The behaviors of the nonautonomous systems will change over time due to external factors
or internal variations, such as the learning models related to neuronal activity in
\cite{GW2012}. To our knowledge, there are few studies on this topic.
Regarding the nonautonomous multi-scale SDEs,
Wainrib \cite{W2013} and Uda \cite{U2021} examined the strong averaging principle
when the coefficients of the fast equation are time-periodic.
Recently, Sun et al. \cite{SWX2024} applied the method of nonautonomous
Poisson equations to investigate the convergence rates in both strong and weak senses. Shi et al. \cite{SLG2025} use the tightness method to investigate the averaging principle for non-autonomous slow-fast McKean-Vlasov SDEs with almost periodic coefficients in the fast component.
When the slow equation's drift and diffusion terms additionally depend on
a highly oscillating time component, \cite{SWX2025,WXXY2024} established the corresponding
averaging principles.
Notably, when the timescales of the highly oscillating time component and the fast process
differ, and the fast equation is autonomous, \cite{CLR2023} showed three
types of averaging principles for SDEs with polynomial nonlinearity.

In the context of the nonautonomous multi-scale SPDEs described in \eref{1.2}, Cerrai and Lunardi in \cite{CL2017} explored the averaging principle for \eref{1.2} when the coefficients
in the fast equation satisfy the almost periodic condition. Li et al. \cite{LSWX2025} study the strong averaging principle for a class of nonautonomous slow-fast SPDEs driven by $\alpha$-stable processes for $\alpha\in(1,2)$.

Therefore, the primary aim of this paper is to study the averaging principle for a large class of nonautonomous multi-scale SPDEs within the (generalized) variational framework, i.e., fully locally monotone and strongly monotone coefficients for the slow and fast equations respectively. Specially, we consider the following multi-scale SPDEs:
\begin{equation}\label{main equation}\left\{\begin{array}{l}
\displaystyle
dX^{\varepsilon}_t=\left[A(t/\varepsilon,X^{\varepsilon}_t)+F(t/\varepsilon,X^{\varepsilon}_t, Y^{\varepsilon}_t)\right]dt
+G_1(t/\varepsilon,X^{\varepsilon}_t)d W^{1}_{t},\\
\displaystyle
dY^{\varepsilon}_t=\frac{1}{\varepsilon}B(t/\varepsilon,X^{\varepsilon}_t, Y^{\varepsilon}_t)dt
+\frac{1}{\sqrt{\varepsilon}}G_2(t/\varepsilon,X^{\varepsilon}_t, Y^{\varepsilon}_t)d W^{2}_{t},\\
X^{\varepsilon}_0=x\in H_1, Y^{\varepsilon}_0=y\in H_2,\end{array}\right.
\end{equation}
where $\varepsilon>0$ is a small parameter describing the ratio of the time scale between the slow component $X^{\varepsilon}_t$
and the fast component $Y^{\varepsilon}_t$, $\{W_t^1\}_{t\geq 0}$ and $\{W_t^2\}_{t\in\mathbb{R}}$ are two cylindrical Wiener process on $U_1$ and $U_2$, respectively.
$A:\RR\times V_1\rightarrow V_1^*$, $F:\RR\times H_1\times H_2\rightarrow H_1$,
$G_1:\RR\times H_1\rightarrow L_2(U_1,H_1)$, $B:\RR\times V_1\times V_2\rightarrow V_2^*$,
$G_2:\RR\times H_1\times H_2\rightarrow L_2(U_2,H_2)$
and the detailed definitions of spaces $H_i,V_i,U_i$ ($i=1,2,$) are presented in the next section. Note that the coefficient $G_1$ is independent of the fast component, which is a crucial setting ensures the validity of the strong averaging principle.

\subsection{Main techniques} \label{subsection 1.2}

Note that in the present setting \eqref{main equation}, the corresponding frozen equation is the following nonautonomous SPDEs:
\begin{equation}
dY^{x,y}_{t}=B(t, x, Y^{x,y}_{t})\,dt+G_2(t,x,Y^{x,y}_{t})\,dW^2_t,\quad x\in H_1.\label{TDFrozenE}
\end{equation}
Instead of using the concept of invariant measure for time-homogeneous semigroup, we need to introduce an evolution system of measures for its time-inhomogeneous semigroup $\{P^x_{s,t}\}_{t\geq s}$ of time-inhomogeneous SPDE \eref{TDFrozenE}, see e.g. \cite{CL2021,DR2006,DR2008}.
Specially, we call $\{\mu^x_t\}_{t\in \RR}$ is an {\em evolution system of measures} for the time-inhomogeneous semigroup $\{P^x_{s,t}\}_{t\geq s}$, if
$$
\int_{H_2}P^x_{s,t} \varphi(y)\,\mu^x_s(dy)=\int_{H_2}\varphi(y)\,\mu^x_t(dy),\quad
s\leq t,\varphi\in C_b(H_2).
$$
To achieve this, we must extend the SDE \eqref{TDFrozenE} to be defined for all $t\in \RR$ rather than just $t\in \RR_+$. This extension requires, in particular, that the noise $W^2_t$ to be well-defined for all $t\in \RR$; see Subsection \ref{sub3.2} for details.

Therefore, it is reasonable to define the averaged coefficients by averaging the drift coefficients $F$ with respect to the evolution system of measures $\{\mu^x_t\}_{t\in \RR}$, i.e.,
\begin{equation}\label{e:drift}
\bar{F}(t,x)=\int_{H_2}F(t,x,y)\,\mu^x_t(dy).
\end{equation}
Thus the corresponding averaged equation can be expressed in the following form:
\begin{equation}\label{AVE1}
d\bar{X}^{\varepsilon}_{t}=\left[A(t/\varepsilon,\bar X^{\varepsilon}_t)+\bar{F}(t/\varepsilon,\bar{X}^{\varepsilon}_t)\right]dt+G_1(t/\varepsilon, \bar{X}^{\varepsilon}_t)\,d W^1_t,\quad\bar{X}^{\varepsilon}_{0}=x.
\end{equation}
Hence, the first purpose of this paper is to prove
 \begin{equation*}
\lim_{\varepsilon \rightarrow 0}\EE \left(\sup_{t\in [0, T]}\|X_{t}^{\varepsilon}-\bar{X}^{\varepsilon}_{t} \|^{2p}_{H_1}\right)=0,\quad \forall p\geq 1, T>0, \label{IR1}
\end{equation*}
where $\bar{X}^{\varepsilon}:=\{\bar{X}^{\varepsilon}_t\}_{t\ge0}$ is the solution to the averaged equation \eqref{AVE1}.
It should be noted that even if we assume that $F$ is time-independent,
the time-dependent $\{\mu_t^x\}_{t\in\RR}$ still implies that $\bar{F}$ constructed in \eqref{e:drift} has a time oscillating component.

\vspace{0.1cm}
Note that the coefficients in the averaged equation \eref{AVE1} retain the dependence of the scaling parameter $\varepsilon$.
Therefore, the second purpose of this paper is to further simplify \eqref{AVE1} by averaging over time $t$.
A crucial step in our analysis is to identify an appropriate function $\bar{F}:H_1\to H_1$ such that
\begin{align}
\lim_{T\rightarrow+\infty}\sup_{t\geq 0}\left\|\frac{1}{T}\int_{t}^{t+T}\bar{F}(s,x)ds-\bar{F}(x)\right\|_{H_1}=0\label{e:Cb}
\end{align}
holds for any $x$ on any bounded subset. This condition is a standard requirement in the studying the averaging principle for stochastic system \eqref{AVE1}, see e.g. \cite[(G1)]{CL2023}.
In order to achieve \eqref{e:Cb},
we will additionally assume that the coefficients $F$, $B$ and $G_2$  satisfy some almost periodic properties
in time $t$. In fact, we will assume that $B$ and $G_2$ are almost periodic in time $t$ uniformly with respect to space
component on any bounded subset (in short, uniformly almost periodic), which implies that
evolution system of measures $\{\mu_t^x\}_{t\in\RR}$ is uniformly almost periodic (see Lemma \ref{Le5.2} for details).
Based on the uniformly almost periodic property of $F$
and $\{\mu_t^x\}_{t\in\RR}$, we obtain that
the averaged coefficient $\bar{F}(t,x)$ in \eref{e:drift} is also uniformly almost periodic (see Lemma \ref{FAP} for details), thus
$$
\bar{F}(x):=\lim_{T\rightarrow+\infty}\frac{1}{T}\int_{t}^{t+T}\bar{F}(s,x)ds
$$
exists and this limit does not depend on $t$ and uniformly with respect to $x$ on any bounded subset; see \cite[Theorem I.3.2]{Besi55} or \cite[Theorem 3.4]{CL2017}.


Additionally, when averaging over time in terms $A(t/\varepsilon,x)$ and $G_1(t/\varepsilon,x)$, we suppose that there exist
$\bar{A}:V_1\rightarrow V_1^*$ and
$\bar{G}_1: H_1\to L_{2}(U_1; H_1)$ such that for any $x\in V_1$,
$\lim_{t\rightarrow+\infty}\|A(t,x)-\bar{A}(x)\|_{V_1^*}=0$, and
for any $t\in\RR$ ,
\begin{align*}
\lim_{T\to+\infty}\frac{1}{T}\int^{t+T}_t \|G_1(s,x)-\bar{G}_1(x)\|^2_{L^2(U_1;H_1)}ds=0.
\end{align*}
Consequently, the corresponding averaged equation can be formulated as follows:
\begin{equation}
d\bar{X}_{t}=\left[\bar{A}(\bar{X}_{t})+\bar{F}(\bar{X}_t)\right]\,dt+\bar{G_1}(\bar{X}_t)\,d W^1_t,\quad\bar{X}_{0}=x. \label{AVE2}
\end{equation}
As a result, we get our second main result:
 \begin{equation*}
\lim_{\varepsilon \rightarrow 0}\EE \left(\sup_{t\in [0, T]}\|X_{t}^{\varepsilon}-\bar{X}_{t} \|^{2p}_{H_1}\right)=0,\quad \forall p\geq 1, T>0,
\end{equation*}
where $\bar{X}$ is the solution to averaged equation \eqref{AVE2}.

\subsection{Contributions}  Despite that the averaging principle was established for
nonautonomous slow-fast systems of stochastic reaction-diffusion equations with almost periodic coefficients in \cite{CL2017},
the mild solution method presented in \cite{CL2017} is not applicable within the variational framework.
So we will utilize techniques from the variational approach to study its averaging principle, which is inspired from \cite{LRSX2023}.
Compared with the convergence in probability in \cite{CL2017},
we obtain the strong convergence in $L^p$ sense for any $p\geq1$.
Meanwhile, although without the almost periodic condition on the coefficients,
we can still obtain a simplified averaged equation,
even whose coefficients still depends on the time scale parameter.
Furthermore, in order to get the strong convergence under the almost periodic condition, we need to show that the evolution system of measures
$\{\mu_t^x\}_{t\in\RR}$ is uniformly almost periodic (see Lemma \ref{Le5.2}), which is also a new result.

Additionally, we will employ the stopping time technique, which is frequently used to overcome challenges arising from nonlinear terms. The main contribution of this paper extends the basic nonautonomous framework investigated by Cheng and Liu in \cite{CL2023} to a fully coupled framework, as well as the autonomous framework explored by Liu et al. in \cite{LRSX2023} to the more general nonautonomous framework. By the way, we improve the locally monotone coefficients discussed in \cite{CL2023,LRSX2023} to the fully local monotone coefficients, thus our results can be applied to a wide range of cases in nonlinear nonautonomous stochastic partial differential equations, such as the nonautonomous stochastic porous medium equation, the nonautonomous stochastic $p$-Laplace equation, the nonautonomous stochastic Burgers type equation and the nonautonomous stochastic 2D Navier-Stokes equation, as well as quasilinear SPDEs, convection diffusion equation, cahn-Hilliard equation, 2D Liquid crystal model as discussed in \cite[Section 4]{RSZ2024}. It is also noted that the averaging principle for autonomous multi-scale SPDEs with fully local monotone coefficients has been established in \cite{HLYY2025} recently, however we here mainly focus on the nonautonomous framework.

\vspace{1mm}
The rest of the paper is organized as follows. In Section \ref{Sec2},
we first present some notations and suitable assumptions, then we formulate our main result in Section \ref{Sec2.1}.
And in Section \ref{Sec2.2}, we will give some examples to illustrate the wide applicability of our result.
In Section \ref{Sec3}, we give some apriori estimates of the solution and study the evolution system of measures of the nonautonomous SPDE. Section \ref{Sec4} is devoted to proving the main results.
Throughout the paper, $C$, $C_T$ and $C_{p,T}$ denote positive constants which may change from line to line, where $C_T$ and $C_{p,T}$ are
used to emphasize that the constants depend on $T$ and $p,T$ respectively.

\section{Main results and examples} \label{Sec2}

\subsection{Assumptions and main results}\label{Sec2.1}
We first present some notations. For $i=1,2$, let $(H_i, \|\cdot\|_{H_i})$ be a separable Hilbert spaces with inner product $\langle\cdot,\cdot\rangle_{H_i}$ and $H^{*}_i$ its dual. Let $(V_i, \|\cdot\|_{V_i})$ be a reflexive Banach space, such that $V_i\subseteq H_i$ continuously and densely. Then for its dual space $V^{*}_i$ it follows that $H^{*}_i\subseteq V^{*}_i$ continuously and densely. Identifying $H_i$ and $H^{*}_i$ via the Riesz isomorphism we have that
$$
V_i\subseteq H_i\equiv H^{*}_i\subseteq V^{*}_i
$$
is a Gelfand triple. Let $_{V^{*}_i}\langle\cdot, \cdot\rangle_{V_i}$ be the dualization between $V^{*}_i$ and $V_i$. Then it follows that
$$
_{V^{*}_i}\langle z_i, v_i\rangle_{V_i} =\langle z_i,v_i\rangle_{H_i},\quad \text{for all}~z_i\in H_i, v_i\in V_i.
$$

\vspace{0.1cm}
For $i=1,2$, let $\{W^{i}_t\}_{t\geq0}$ be a cylindrical $\mathscr{F}_t$-Wiener process in a separable Hilbert space $(U_i, \|\cdot\|_{U_i})$ on a probability space $(\Omega,\mathscr{F},\mathbb{P})$ with natural filtration $\mathscr{F}_{t}$, that is,
$$
W^{i}_t=\sum_{k\in\mathbb{N}_{+}}W^{i,k}_te_k,
$$
where $\{e_{i,k}\}_{k\in \mathbb{N}_{+}}$ is an orthonormal basis of $U_i$ and $\{W^{i,k}_t\}_{k\in\mathbb{N}_{+}}$ is a sequences of independent one dimensional standard Brownian motion. Let $L_{2}(U_i,H_i)$ be the space of Hilbert-Schmidt operators from $U_i$ to $H_i$.
The norm on $L_{2}(U_i,H_i)$ is defined by
$$\|S\|^2_{L_{2}(U_i,H_i)}:=\sum_{k\in \mathbb{N}_{+}}\|S e_{i,k}\|^2_{H_i},\quad S\in L_{2}(U_i,H_i),$$
We also assume the processes $\{W^{1}_t\}_{t\geq0}$ and $\{W^{2}_t\}_{t\geq0}$ are independent.
Since our primary focus is on almost periodic nonautonomous
multi-scale SPDEs, we assume all coefficients are well-defined for all $t\in\RR$.
Moreover, assume that
the coefficients
$$
A: \RR\times V_1\to V^{*}_1,\quad F:\RR\times H_1\times H_2\to H_1,
\quad G_1: \RR\times H_1\to L_{2}(U_1; H_1),
$$
and
$$
B: \RR\times H_1\times V_2\to V^{*}_2, \quad
G_2:\RR\times H_1\times H_2\to L_{2}(U_2; H_2)
$$
are measurable.

We recall the nonautonomous multi-scale SPDEs \eqref{main equation}.
For the coefficients of the slow equation in \eref{main equation}, we suppose that there exist constants $\alpha\in(1,\infty)$, $\beta\in [0, \infty)$, $\theta\in(0,\infty)$ and $C>0$ such that the following
conditions hold for all $u,v,w\in V_1$, $u_1,v_1\in H_1$, $u_2,v_2\in H_2$ and $t\in\RR$:

\begin{conditionA}(Hemicontinuity)\label{A1}
The map $\lambda\mapsto{_{V^{*}_1}}\langle  A(t,u+\lambda v), w\rangle_{V_1}$ is continuous on $\RR$.
\end{conditionA}

\begin{conditionA} (Local monotonicity) \label{A2}
\begin{align}
_{V^{*}_1}\langle A(t,u)-A(t,v), u-v\rangle_{V_1}\leq \left[\rho(v)+\eta(u)\right]\|u-v\|^2_{H_1},\label{A21}
\end{align}
where $\rho,\eta: V_1\to [0,\infty)$ are two measurable functions satisfying
\begin{eqnarray*}
\rho(v)+\eta(v)\leq C(1+\|v\|^{\alpha}_{V_1})(1+\|v\|^{\beta}_{H_1}).
\end{eqnarray*}
\end{conditionA}

\begin{conditionA} (Coercivity)\label{A3}
\begin{align*}
_{V^{*}_1}\langle A(t,v), v\rangle_{V_1}\leq C\|v\|^2_{H_1}-\theta\|v\|^{\alpha}_{V_1}+C.
\end{align*}
\end{conditionA}

\begin{conditionA}(Growth)\label{A4}
\begin{eqnarray*}
\|A(t,v)\|^{\frac{\alpha}{\alpha-1}}_{V^{*}_1}\leq C(1+\|v\|^{\alpha}_{V_1})(1+\|v\|^{\beta}_{H_1}).
\end{eqnarray*}
\end{conditionA}

\begin{conditionA}(Lipschitz)\label{A5}

\begin{align}\begin{split}\label{0525:01}
&\|F(t,u_1,u_2)-F(t,v_1,v_2)\|_{H_1}\leq C(\|u_1-v_1\|_{H_1}+\|u_2-v_2\|_{H_2}),\\
&\|F(t,u_1,u_2)\|_{H_1}\leq C(1+\|u_1\|_{H_1}+\|u_2\|_{H_2}),\\
&\|G_1(t,u_1)-G_1(t,v_1)\|_{L_2(U_1,H_1)}\leq C\|u_1-v_1\|_{H_1},\\
&\|G_1(t,u_1)\|_{L_2(U_1,H_1)}\leq C(1+\|u_1\|_{H_1}).
\end{split}
\end{align}
\end{conditionA}


For the coefficients of the fast equation in \eref{main equation}, we suppose that there exist constants
$\kappa\in (1,\infty)$, $\gamma,\eta\in (0, \infty)$, $\zeta\in(0,1)$
and $C>0$ such that the following
conditions hold for all $u,v,w\in V_2$, $u_1,v_1\in H_1$,
$u_2\in H_2$ and $t\in\mathbb{R}$:
\begin{conditionB}(Hemicontinuity)\label{B1}
The map $\lambda\mapsto{_{V_{2}^*}}\langle B(t,u_1+\lambda v_1,u+\lambda v), w\rangle_{V_2}$ is continuous on $\RR$.
\end{conditionB}

\begin{conditionB} (Strong monotonicity) \label{B2}
\begin{align}
&
2_{V^{*}_2}\langle B(t,u_1,u)-B(t,v_1,v), u-v\rangle_{V_2}+\|G_2(t,u_1,u)-G_2(t,v_1,v)\|^2_{L_2(U_2,H_2)}\nonumber\\
\leq&
-\gamma\|u-v\|^2_{H_2}
+C\|u_1-v_1\|^2_{H_1}.\label{dc}
\end{align}
\end{conditionB}

\begin{conditionB} (Coercivity)\label{B3}
\begin{align*}
_{V^{*}_2}\langle B(t,u_1,v), v\rangle_{V_2}\leq C\|v\|^2_{H_2}-\eta\|v\|^{\kappa}_{V_2}
+C(1+\|u_1\|^2_{H_1}).
\end{align*}
\end{conditionB}

\begin{conditionB}(Growth)\label{B4}
\begin{eqnarray*}
\|B(t,u_1,v)\|_{V^{*}_2}\leq C\left(1+\|v\|^{\kappa-1}_{V_2}+\|u_1\|^{\frac{2(\kappa-1)}{\kappa}}_{H_1}\right)
\end{eqnarray*}
and
\begin{eqnarray}
 \|G_2(t,u_1,u_2)\|_{L_2(U_2,H_2)}\leq C(1+\|u_1\|_{H_1}+\|u_2\|^{\zeta}_{H_2}).\label{G2}
\end{eqnarray}
\end{conditionB}


\begin{remark} We here give some comments for the assumptions above.

(i) Condition \eref{A21} is a fully local monotonicity condition, which is weaker than the usual local monotonicity condition used in \cite{LR2015}. The major difference is that in condition \eref{A21} both $\rho$ and $\eta$ can be nonzero, thus it will cover more examples (see \cite[Section 4]{RSZ2024}).

(ii) Condition \eref{dc} is also called the dissipativity condition,
which guarantees that evolution system of measure of the nonautonomous
frozen equation is unique and its exponentially ergodicity holds.

(iii) The assumption $\zeta\in(0,1)$ in condition \eref{G2} is used to ensure the solution $(X^{\vare}_t, Y^{\vare}_t)$ has finite $p$-th moments for any $p>0$.

\end{remark}

\medskip
Now, we recall the following definition of a variational solution in \cite{LR2015,RSZ2024}.

\begin{definition}\label{S.S.}
For any given $\vare>0$, an $H_1\times H_2$-valued continuous and $\mathscr{F}_{t}$-adapted
process $(X^{\vare}_t, Y^{\vare}_t)_{t\in [0, T]}$ is called a solution of
system \eqref{main equation}, if for its $dt\otimes \PP$-equivalence class
$(\check{X}^{\vare}, \check{Y}^{\vare})$ we have
$\check{X}^{\vare}\in L^{\alpha}([0, T]\times \Omega, dt\otimes \PP; V_1)\cap L^2([0, T]\times\Omega, dt\otimes \PP; H_1)$
with $\alpha$ as in \ref{A3}, $\check{Y}^{\vare}\in L^{\kappa}([0, T]\times \Omega, dt\otimes \PP; V_2)\cap L^2([0, T]\times\Omega, dt\otimes \PP; H_2)$
with $\kappa$ as in \ref{B3} and $\PP$-a.s.
\begin{equation}\label{mild solution}
\begin{cases}
X^{\varepsilon}_t=X^{\varepsilon}_0+\int^t_0 A(s/\varepsilon,\tilde{X}^{\varepsilon}_s)ds+\int^t_0 F(s/\varepsilon,X^{\varepsilon}_s, Y^{\varepsilon}_s)ds+\int^t_0 G_1(s/\varepsilon,X^{\varepsilon}_s)dW^{1}_s,\\
Y^{\varepsilon}_t=Y^{\varepsilon}_0+\frac{1}{\varepsilon}\int^t_0 B(s/\varepsilon,X^{\varepsilon}_s, \tilde{Y}^{\varepsilon}_s)ds
+\frac{1}{\sqrt{\varepsilon}}\int^t_0 G_2(s/\varepsilon,X^{\varepsilon}_s, Y^{\varepsilon}_s)dW^{2}_s,
\end{cases}
\end{equation}
where $(\tilde{X}^{\varepsilon}, \tilde{Y}^{\varepsilon})$ is any $V_1\times V_2$-valued progressively measurable $dt\otimes \PP$-version of $(\check{X}^{\vare},\check{Y}^{\vare})$.
\end{definition}

Using the similar argument in the proofs of \cite[Theorem 2.6]{RSZ2024} and \cite[Theorem 2.3]{LRSX2023}, we can easily obtain the following well-posedness result. The detailed proof is omitted.

\begin{lemma}\label{Th1}
Assume the conditions \ref{A1}-\ref{A5}, \ref{B1}-\ref{B4} hold. Then for any $\vare>0$ and initial values $(x, y)\in H_1\times H_2$, the system \eqref{main equation} has a unique solution $\{(X^{\varepsilon}_t,Y^{\varepsilon}_t)\}_{t\geq0}$.
\end{lemma}

%
%

\vspace{0.2cm}
Now, we present our first main result:
\begin{theorem}\label{main result 1}
Assume the conditions \ref{A1}-\ref{A5} and \ref{B1}-\ref{B4} hold. Then for any initial values $(x, y)\in H_1\times H_2$, $p\geq1$ and $T>0$, we have
\begin{align}
\lim_{\vare\rightarrow 0}\mathbb{E} \left(\sup_{t\in[0,T]}\|X_{t}^{\vare}-\bar{X}^{\varepsilon}_{t}\|^{2p}_{H_1} \right)=0,\label{2.2}
\end{align}
where $\{X_t^\varepsilon\}_{t\ge0}$ and $\{\bar X^{\vare}_t\}_{t\ge0}$ are the variational solutions to the SPDEs  \eqref{main equation} and \eqref{AVE1} respectively.
\end{theorem}

\vspace{0.1cm}

Note that the coefficients in the averaged equation \eref{AVE1} maintain the dependence on
the scaling parameter $\varepsilon$. To achieve further effective system, we proceed to
perform an additional averaging of \eqref{AVE1} with respect to time $t$.
In the subsequent analysis, we will focus on cases where the functions
$F$, $B$ and $G_2$ have some recurrent properties, while $A$ and $G_1$
exhibit certain asymptotically autonomous behavior.
As preparation for this analysis, we first recall the definition of almost periodic functions.
\begin{definition}Let $(\mathcal{X},d)$ be a complete metric space.
We call $\phi:\RR\rightarrow \mathcal{X}$ is a {\it almost periodic function}, if for any $\epsilon>0$, there exists a constant
$l=l(\epsilon)>0$ such that $\mathcal T(\varphi,\epsilon)\cap[a,a+l]\neq\varnothing$ for all $a\in\RR$,
where
$$
\mathcal T(\varphi,\epsilon):=\left\{\tau\in\RR:\sup_{t\in\RR}
d(\varphi(t+\tau),\varphi(t))<\epsilon\right\}.
$$
\end{definition}
\begin{remark}
The notion of almost periodic functions was proposed and comprehensively studied by Bohr \cite{Bohr}.
In Celestial Mechanics, it is widely recognized that almost periodic solutions and
stable solutions are fundamentally interconnected.
Similarly, stable electronic circuits exhibit almost periodic behavior. Set $\mathcal{AP}$ be the space of all almost periodic functions $\phi:\RR\rightarrow\RR$ under the usual Euclidean distance. In the following, we give some special examples of almost periodic functions:

(1) Define $\phi_1(t):=\sin(\rho_1t)+\cos(\rho_2t)$. If $\rho_1/\rho_2$ is rational
then $\phi_1$ is periodic. Otherwise, $\phi_1$ is almost periodic and not periodic. Actually, if $\rho_1/\rho_2$ is irrational, $\phi_1$ is quasi-periodic.

(2) Define $\phi_2(t):=\sum_{n=1}^{\infty}\frac{1}{n^2}\sin(\sqrt{n}t)$, then $\phi_2\in\mathcal{AP}$.\\
Both periodic and quasi-periodic functions fall under the category of almost periodic functions.
For the sake of brevity, we will not separately introduce quasi-periodic functions in this paper.
\end{remark}

Note that the coefficients $F$, $B$ and $G_2$ still depend on the space variables, thus we need the concept of uniformly almost periodic.

\begin{definition}\rm
Let $(\mathcal{X}_1,d_1)$ and $(\mathcal{X}_2,d_2)$ be complete metric spaces.
We say
$\varphi\in C(\RR\times\mathcal{X}_1,\mathcal{X}_2)$ is {\it uniformly almost periodic} in $t$
if $\varphi$ is bounded on every bounded subset from $\RR\times\mathcal X_1$ and
for any $\epsilon>0$ and any bounded subset $Q\subset \mathcal X_1$, there exists a constant
$l=l(\epsilon,Q)>0$ such that $\mathcal T(\varphi,\epsilon,Q)\cap[a,a+l]\neq\varnothing$ for all $a\in\RR$,
where
$$
\mathcal T(\varphi,\epsilon,Q):=\left\{\tau\in\RR:\sup_{(t,x)\in\RR\times Q}
d_2(\varphi(t+\tau,x),\varphi(t,x))<\epsilon\right\}.
$$
\end{definition}

\begin{remark} Suppose that $F:\RR\times\mathcal{X}_1\rightarrow\mathcal X_2$ is continuous in $t$ uniformly with respect to $x_1$
on every bounded subset $Q\subset\mathcal X_1$ and bounded on every bounded subset from $\RR\times\mathcal X_1$. We here give a sufficient and necessary condition for ensuring
$F$ is uniformly almost periodic, that is, if and only if for any $\{t_n'\}\subset \RR$ there exists a
subsequence $\{t_n\}\subset\{t_n'\}$ such that 
\begin{equation*}
\lim_{n\rightarrow\infty}F(t+t_n,x_1)=:F'(t,x_1)
\end{equation*}
exists uniformly with respect to $t\in\RR$ and $x_1\in Q$; see \cite{CL2020, CL2023}. In the following, we give several special examples:

(1) Let $(\mathcal X_1,\|\cdot\|)$ be normed vector space.
Set
\begin{equation*}
F(t,x_1):=\phi(t)f(x_1),
\end{equation*}
where $\phi\in\mathcal{AP}$, $f:\mathcal X_1\rightarrow \mathcal X_2$ is bounded on any bounded subset.
Then $F$ is uniformly almost periodic.
Indeed, for any $\{t_n'\}\subset \RR$ there exists $\{t_n\}\subset\{t_n'\}$
and $\phi'$ such that
$\lim_{n\rightarrow\infty}\sup_{t\in\RR}|\phi(t+t_n)-\phi'(t)|=0$. Then
for any bounded $Q\subset \mathcal X_1$,
\begin{align*}
\lim_{n\rightarrow\infty}\sup_{t\in\RR,x_1\in Q}
\left\|\phi(t+t_n)f(x_1)-\phi'(t)f(x_1)\right\|
&
=\lim_{n\rightarrow\infty}\sup_{t\in\RR,x_1\in Q}|\phi(t+t_n)-\phi'(t)|
\,\|f(x_1)\|=0.
\end{align*}

(2) We can also consider some examples formulated as Nemytskii operators.
Let $\Lambda\subset\RR^n$ be a  bounded subset,
and suppose $f:\RR\times\Lambda\times \RR\rightarrow\RR$ is almost periodic in $t$
uniformly with respect to $(x,\xi)$ on $\Lambda\times\RR$.
Assume further that $f$ is linear growth with respect to $\xi$.
Then the mapping $F:\RR\times L^2(\Lambda)\rightarrow L^2(\Lambda)$ defined by
$F(t,u)(x):=f(t,x,u(x))$ is continuous on $L^2(\Lambda)$ and uniformly almost periodic.
A concrete example is the mapping $F:\RR\times L^2([0,1])\rightarrow L^2([0,1]),~
F(t,u)(x):=\sin(t+x+u(x))$, which is uniformly almost periodic.
\end{remark}

Now we impose the following conditions:

\begin{conditionC}\label{C1}
Suppose that there exists $\bar{A}:V_1\rightarrow V_1^*$, which satisfies \ref{A1},
such that for all $x\in V_1$,
$$
\lim_{t\rightarrow\infty}\|A(t,x)-\bar{A}(x)\|_{V_1^*}=0,
$$
and there exists $\bar{G}_1: H_1\to L_{2}(U_1, H_1)$ such that for any $T\geq 0$, $R>0$ and $x\in H_1$ with $\|x\|_{H_1}\leq R$,
\begin{align}\label{C1eq}
\sup_{t\in \RR}\frac{1}{T}\int^{t+T}_t \|G_1(s,x)-\bar{G}_1(x)\|^2_{L_{2}(U_1, H_1)}ds\leq \phi^R_1(T),
\end{align}
where $\phi^R_1(T)\rightarrow 0$ as $T\to +\infty$.
\end{conditionC}

\begin{conditionC}\label{C2}
Suppose that there exists a closed subset $S\subset H_2$ equipped with the norm $\|\cdot\|_S$ such that
$V_2\subset S$ is continuous and $S\subset H_2$ is compact. Let $T_n$ be a sequence of positive definite
self-adjoint operators on $H_2$ such that for each $n\geq1$,
\begin{equation*}
\langle y_1,y_2\rangle_n:=\langle y_1,T_ny_2\rangle_{H_2},\quad y_1,y_2\in H_2,
\end{equation*}
defines a new inner product on $H_2$. Assume further that the norms $\|\cdot\|_n$ generated by
$\langle\cdot,\cdot\rangle_n$ are all equivalent to $\|\cdot\|_{H_2}$ and for all $y_1\in S$ we have
\begin{equation*}
\|y_1\|_n\uparrow \|y_1\|_S \quad {\text{as}} ~n\rightarrow\infty.
\end{equation*}
Furthermore, we suppose that for each $n\geq1$, $T_n: V_2\rightarrow V_2$ is continuous and there exist
constants $\gamma_1, C>0$ such that for all $x\in H_1$, $y\in V_2$ and $t\in \RR$,
\begin{equation*}
2_{V_2^*}\langle B(t,x,y),T_ny\rangle_{V_2}+\|G_2(t,x,y)\|^{2}_{L_2(U_2,H_2)}
\leq -\gamma_1\|y\|_{n}^2+C(1+\|x\|_{H_1}^2).
\end{equation*}
\end{conditionC}

\begin{conditionC}\label{C3}
Assume that $F$, $B$ and $G_2$ are uniformly almost periodic. Moreover, $B(t,x,y)=B_1(x,y)+B_2(t,x,y)$, where
$B_1:H_1\times V_2\rightarrow V_2^*$ and $B_2:\RR\times H_1\times H_2\rightarrow H_2$ satisfy
that there exists a constant $C>0$ such that
for all $t\in\RR$, $x\in H_1$ and $y\in H_2$,
$$
\|B_2(t,x,y)\|_{H_2}\leq C(1+\|x\|_{H_1}+\|y\|_{H_2}).
$$
\end{conditionC}

\begin{remark}\label{Rem0722}
Condition \ref{C1} is used to average over time in terms $A(t/\varepsilon,x)$ and $G_1(t/\varepsilon,x)$. 
It follows from \ref{C1} that $\bar{A}$ satisfies conditions \ref{A2}, \ref{A3} and \ref{A4}.
As demonstrated in Examples \ref{ExCHe} and \ref{2DLeq} below, these terms satisfy \ref{C1}.
Condition \ref{C2} is used to prove that the evolution system of measures
$\{\mu^x_t\}_{t\in\RR}$ is {\it almost periodic} for any $x\in H_1$. We point out that certain stochastic reaction-diffusion equations and stochastic porous medium equations satisfy \ref{C2}; see e.g. \cite{CL2021, CL2023} for details. While, condition \ref{C3} further ensures that $\{\mu^x_t\}_{t\in\RR}$ is {\it uniformly almost periodic}; see Lemma \ref{Le5.2}.
\end{remark}

Now, we are in a position to present our second main result.
\begin{theorem}\label{main result 2}
Assume the conditions \ref{A1}-\ref{A5}, \ref{B1}-\ref{B4} and \ref{C1}-\ref{C3} hold.
Then for any initial values $(x, y)\in H_1\times H_2$, $p\geq1$ and $T>0$, we have
\begin{align}
\lim_{\vare\rightarrow 0}\EE\left(\sup_{t\in [0, T]}\|X_{t}^{\varepsilon}-\bar{X}_{t} \|^{2p}_{H_1}\right)=0,\label{2.3}
\end{align}
where $\{X_t^\varepsilon\}_{t\ge0}$ and $\{\bar X_t\}_{t\ge0}$  are the variational solutions to the SPDEs  \eqref{main equation} and \eqref{AVE2} respectively.
\end{theorem}

\subsection{Application to examples}\label{Sec2.2}
In this subsection, we will illustrate our theoretical results by some examples.
We point out that all the examples considered in \cite{LRSX2023} can be covered by our framework,
including the stochastic porous media equation, $p$-Laplace equation, Burgers equation
and 2D Navier-Stokes equation.
For brevity, we just introduce two examples which did not be covered in the existing works,
and  consider the linear multiplicative noise in these examples.

Let $D$ be bounded domain in $\RR^d$ with smooth boundary $\partial D$.
We will use $L^p(D)$ to denote the space of $p$-Lebesgue integrable functions on $D$.
Denote by $H^p(D)$ ($H^p_0(D)$) the usual Sobolev space on $D$ (with zero trace).
Note that we mainly focus on the nonlinear operator $A$, so
assume that $F:\RR\times H_1\times H_2\rightarrow H_1$ satisfies \ref{A5}.
Let $\Xi$ be the set of all mapping $\ell:\RR\rightarrow\RR_+$ satisfing $\lim_{t\rightarrow+\infty}\ell(t)=\bar{\ell}$
with $\bar{\ell}>0$, For example, $\ell(t):=\frac{1}{1+|t|^\iota}+1$ with $\iota>0$.

\begin{example}[Stochastic Cahn-Hilliard equation]\label{ExCHe}
Consider the following slow-fast stochastic Cahn-Hilliard-heat equation
\begin{equation}\label{0312:01}
\begin{cases}
du^\varepsilon_t=\left(-\ell_1(t/\varepsilon)\Delta^2u_t^\varepsilon+F(t/\varepsilon,u_t^\varepsilon,v_t^\varepsilon)\right)dt
+\ell_2(t/\vare)u^\varepsilon_tdW^1_t,\quad u_0^\vare=u\in L^2(D),
\\
dv^\varepsilon_t=\frac{1}{\varepsilon}\left(\Delta v_t^\varepsilon
+\phi(t/\varepsilon)v_t^\varepsilon+u_t^\varepsilon\right)dt
+\frac{1}{\sqrt{\varepsilon}}cv_t^\vare dW_t^2,\quad v_0^\vare=v\in L^2(D),
\end{cases}
\end{equation}
where $\ell_1,\ell_2\in\Xi$, $c\in\RR$ and $|\phi|_\infty+\tfrac{c^2}{2}<\lambda_*$.
Here $\lambda_*$ is the first eigenvalue of $-\Delta$ on
$L^2(D)$ with Dirichlet boundary condition.
Define
$$
V_1=\left\{u\in H^2(D):\nabla u\cdot \nu=\nabla (\Delta u)\cdot\nu=0 {\text{ on }} \partial D  \right\},
\quad H_1=L^2\left(D\right)
$$
and
$$
V_2:= H_0^{1}\left(D\right), \quad H_2:=L^2\left(D\right).
$$
Note that conditions \ref{B1}-\ref{B4} hold; see \cite[Section 6.1]{CL2023} for details.
And it follows from Theorems 3.6 and 3.9 in
\cite{CL2023} that for any $u\in L^2(D)$,
there exists a unique evolution system of measures $\{\mu_t^u\}_{t\in\RR}$ to
\begin{equation*}
dv_t^u=\left(\Delta v_t^u+\phi(t)v_t^u+u\right)dt+cv_t^ud\bar{W}_t^2.
\end{equation*}
Let $A(t,u):=-\ell_1(t)\Delta^2u$.
By \cite[Example4.4]{RSZ2024}, one sees that conditions \ref{A1}-\ref{A4} hold.
Then in view of Theorem \ref{main result 1}, we have for any $T>0$,
\begin{equation*}
\lim_{\vare\rightarrow0}\EE\left(\sup_{t\in[0,T]}\|u_t^\vare-\bar{u}_t^\vare\|_{L^2(D)}^2\right)=0,
\end{equation*}
where $\bar{u}_t^\vare$ is the solution for
\begin{equation*}
d\bar{u}_t^\vare=\left(-\ell_1(t/\vare) \Delta^2\bar{u}_t^\vare+\bar{F}(t/\vare,\bar{u}_t^\vare)\right)dt
+\ell_2(t/\vare)\bar{u}_t^\vare dW_t^1,\quad
\bar{u}_0^\vare=u\in L^2(D).
\end{equation*}
Here $\bar{F}(t,u):=\int_{H_2}F(t,u,v)\mu_t^u(dv)$.
Furthermore, assume that $F$ is uniformly almost periodic and $\phi\in\mathcal{AP}$.
Since $\ell\in \Xi$, \ref{C1} holds.
It is clear that  condition \ref{C3} holds. It follows from \cite[Section 6.1]{CL2023} that \ref{C2} holds.
Then by Theorem \ref{main result 2}, we have for any $T>0$,
\begin{equation*}
\lim_{\varepsilon\rightarrow0}\EE\left(\sup_{t\in[0,T]}\|u_t^\varepsilon-\bar{u}_t\|_{L^2(D)}^2\right)=0,
\end{equation*}
where $\bar{u}_t$ is the solution to
\begin{equation*}
d\bar{u}_t=\left(-\bar{\ell}_1 \Delta^2\bar{u}_t+\bar{F}(\bar{u}_t)\right)dt
+\bar{\ell}_2\bar{u}_tdW_t^1,\quad
\bar{u}_0=u\in L^2(D).
\end{equation*}
Here $\bar{F}(u):=\lim_{T\rightarrow\infty}\frac{1}{T}\int_0^T\int_{H_2}F(t,u,v)\mu_t^u(dv)$dt.
\end{example}

\begin{example}[Stochastic 2D liquid crystal model]\label{2DLeq}
Set $D\subset \RR^2$ and define
$$
\widetilde{V}:=\left\{u\in H^1(D)^2: \nabla\cdot u=0, u|_{\partial D}=0\right\}.
$$
Let $\widetilde{H}$ be the closed of $\widetilde{V}$ under the $L^2$-norm $\|u\|_{\widetilde{H}}^2:=\int_{D}|u(x)|^2dx$
and
$$
P_H:L^2(D)^2\rightarrow \widetilde{H}
$$
be the usual Helmholtz-Leray projection. Now we set $x:=(u,n)$,
$$
V_1:=\widetilde{V}\times \left\{n\in H^2(D)^3:\frac{\partial n}{\partial\nu}=0\right\},\quad
H_1:=\widetilde{H}\times [H^1(D)^3]
$$
with the norms denoted by
$$
\|x\|_{V_1}^2:=\|u\|_{\widetilde{V}}^2+\|n\|_{H_0^2}^2, \quad
\|x\|_{H_1}^2:=\|u\|_{\widetilde{H}}^2+\|n\|_{H^1}^2,
$$
and let
\begin{equation*}
A(x):=
\begin{pmatrix}
P_H\left[\Delta u-(u\cdot\nabla)u-\nabla n\cdot \Delta n\right]\\
\Delta n-(u\cdot\nabla)n-\Phi(n)
\end{pmatrix}
,
\end{equation*}
where $\Phi(n)=\left(\sum_{i=0}^{k}a_i|n|^{2i}\right)n$ with some $k\in\NN$. Here $a_i\in\RR$
for $i=0,1,...,k-1$ and $a_k>0$. Set
$$
V_2:=L^p(D) \quad \text{and} \quad H_2:=H^{-1}_0(D),
$$
where $H_0^{-1}(D)$ is the dual space of $H_0^1(D)$.

Consider the following slow-fast stochastic liquid-crystal-porous-media equation
\begin{equation}\label{0317:01}
\begin{cases}
dX_t^\varepsilon=\left(\ell_1(t/\varepsilon)A(X_t^\varepsilon)+F(t/\varepsilon,X_t^\varepsilon,Y_t^\varepsilon)\right)dt
+\ell_2(t/\vare)X_t^\varepsilon dW_t^1,\quad X_0^\vare=x\in H_1,\\
dY_t^\varepsilon=\frac{1}{\varepsilon}\left[\Delta\left(|Y_t^\varepsilon|^{p-2}Y_t^\varepsilon
+aY_t^\varepsilon\right)+\phi(t/\varepsilon)Y_t^\varepsilon\right]dt
+\frac{1}{\sqrt{\varepsilon}}G_2 dW_t^2, \quad Y_0^\vare=y\in H_2,
\end{cases}
\end{equation}
where $a>0$, $\ell_1,\ell_2\in\Xi$, $p>2$, $\phi:\RR\rightarrow(-\infty,0)$ and $G_2\in L_2(L^p(D))$.
It follows from \cite[Theorem 6.3]{CL2023} that
conditions \ref{B1}-\ref{B4} hold, which implies that
there exists a unique evolution system of measures $\{\mu_t\}_{t\in\RR}$ to
\begin{equation*}
dY_t=\left[\Delta\left(|Y_t|^{p-2}Y_t
+aY_t\right)+\phi(t)Y_t\right]dt+G_2d\bar{W}_t^2.
\end{equation*}
Define
\begin{equation*}
\bar{F}(t,x):=\int_{H_2}F(t,x,y)\mu_t(dy).
\end{equation*}
It can be verified that $A$ satisfies conditions \ref{A1}-\ref{A4} (see \cite[Example 4.5]{RSZ2024}). Then with the help of Theorem \ref{main result 1}, we obtain for any $T>0$,
\begin{equation*}
\lim_{\vare\rightarrow0}\EE\left(\sup_{t\in[0,T]}\|X_t^\vare-\bar{X}_t^\vare\|_{H_1}^2\right)=0,
\end{equation*}
where $\bar{X}_t^\vare$ is the solution to
\begin{equation*}
dX_t^\varepsilon=\left(\ell_1(t/\varepsilon)A(X_t^\varepsilon)+\bar{F}(t/\varepsilon,X_t^\varepsilon)\right)dt
+\ell_2(t/\vare)X_t^\varepsilon dW_t^1,\quad X_0^\vare=x\in H_1.
\end{equation*}

Furthermore, assume that $F$ is uniformly almost periodic and $\phi\in\mathcal{AP}$.
It can be verified that conditions \ref{C1} and \ref{C3} hold.
By the proof of \cite[Theorem 6.3]{CL2023}, one sees that \ref{C2} holds.
Then employing Theorem \ref{main result 2},
we have for any $T>0$,
\begin{equation*}
\lim_{\vare\rightarrow0}\EE\left(\sup_{t\in[0,T]}\|X_t^\varepsilon-\bar{X}_t\|_{H_1}^2\right)=0,
\end{equation*}
where $\bar{X}_t$ is the solution to
$$
d\bar{X}_t=\left(\bar{\ell}_1 A(X_t^\varepsilon)+\bar{F}(\bar{X}_t)\right)dt
+\bar{\ell}_1\bar{X}_t dW_t^1.
$$
Here,
\begin{equation*}
\bar{F}(x):=\lim_{T\rightarrow\infty}\frac{1}{T}\int_0^T\int_{H_2}F(t,x,y)\mu_t(dy)dt.
\end{equation*}

\end{example}

\section{Apriori estimates and evolution system of measures} \label{Sec3}

In this section, we first establish a priori estimates for the solution process
$\{(X^{\varepsilon}_t,Y^{\varepsilon}_t)\}_{t\geq0}$.
These estimates lead to crucial bounds on the temporal increments of
$\{X_{t}^{\varepsilon}\}_{t\geq0}$, which are fundamental for proving our main results.
Subsequently, we employ a time discretization technique to construct an auxiliary process
$\{\hat{Y}_{t}^{\varepsilon}\}_{t\geq0}$ and analyze the difference process
$\{Y^{\varepsilon}_t-\hat{Y}_{t}^{\varepsilon}\}_{t\geq0}$. Finally, we investigate
the evolution system associated with the corresponding non-autonomous frozen equation.
For clarity and consistency, we maintain conditions \ref{A1}-\ref{A5} and
\ref{B1}-\ref{B4} throughout this section, and fix an initial value
$(x,y)\in H_1\times H_2$.

\subsection{Apriori estimates}

Note that all the conditions \ref{A1}-\ref{A5} and
\ref{B1}-\ref{B4} hold uniformly for $t\in\RR$. Therefore, the proofs of the following Lemmas \ref{PMY}, \ref{COX} and \ref{DEY} are very similar to the arguments presented in \cite{LRSX2023}, and we will omit the detailed proofs for the sake of simplicity.

\begin{lemma} \label{PMY}
For any  $T>0$ and $p\geq1$, there exists a constant $C_{p,T}>0$ such that
\begin{align}
\sup_{\vare\in (0,1)}\mathbb{E}\left(\sup_{t\in[0,T]}\|X_{t}^{\vare}\|^{2p}_{H_1}\right)
+\sup_{\vare\in (0,1)}\EE\left(\int^T_0\|X_{t}^{\vare}\|^{2p-2}_{H_1}
\|\tilde{X}_{t}^{\vare}\|^{\alpha}_{V_1}dt\right)
\leq C_{p,T}\left(1+\|x\|^{2p}_{H_1}+\|y\|^{2p}_{H_2}\right)\label{F3.1}
\end{align}
and
\begin{align}
\sup_{\vare\in (0,1)}\sup_{t\in[0, T]}\mathbb{E}\|Y_{t}^{\varepsilon}\|^{2p}_{H_2}\leq C_{p,T}\left(1+\|x\|^{2p}_{H_1}+\|y\|^{2p}_{H_2}\right).\label{E3.2}
\end{align}
\end{lemma}

\vspace{0.1cm}
\begin{lemma} \label{COX}
For any $T>0$, there exist constants $m>0$ and $C_{T}>0$ such that such that for all $\vare\in(0,1)$ and $\delta>0$ small enough,
\begin{align}\label{F3.7}
\mathbb{E}\left[\int^{T}_0\|X_{t}^{\varepsilon}-X_{t(\delta)}^{\varepsilon}\|^2_{H_1} dt\right]\leq C_{T}(1+\|x\|^m_{H_1}+\|y\|^m_{H_2})\delta^{1/2},
\end{align}
where $t(\delta):=[t/\delta]\delta$ and $[s]$ denotes the largest integer which is smaller than $s$.
\end{lemma}

\vspace{0.1cm}
Inspired from the method of time discretization used in \cite{K1968},  we first construct an auxiliary process $\hat{Y}_{t}^{\varepsilon}\in H_2$ satisfying the following equation
$$
d\hat{Y}_{t}^{\vare}=\frac{1}{\vare}B\left(t/\vare,X^{\vare}_{t(\delta)},\hat{Y}_{t}^{\vare}\right)dt
+\frac{1}{\sqrt{\vare}}G_2\left(t/\vare,X^{\vare}_{t(\delta)},\hat{Y}_{t}^{\vare}\right)dW^{2}_t,\quad \hat{Y}_{0}^{\vare}=y\in H_2,
$$
where $\delta$ is a fixed positive number depending on $\vare$ and will be chosen later. Then for its $dt\otimes \PP$-equivalence class $\check{\hat{Y}}^{\vare}$ we have $\check{\hat{Y}}^{\vare}\in L^{\kappa}([0, T]\times \Omega, dt\otimes \PP; V_2)\cap L^2([0, T]\times\Omega, dt\otimes \PP; H_2)$ with $\kappa$ as in \ref{B3}, and for any $k\in \mathbb{N}$ and $t\in[k\delta,\min((k+1)\delta,T)]$, $\PP$-a.s.
\begin{align*}
\hat{Y}_{t}^{\varepsilon}=\hat{Y}_{k\delta}^{\varepsilon}+\frac{1}{\varepsilon}\int_{k\delta}^{t}
B(s/\vare,X_{k\delta}^{\varepsilon},\hat{Y}_{s}^{\varepsilon})ds
+\frac{1}{\sqrt{\varepsilon}}\int_{k\delta}^{t}G_2(s/\vare,X_{k\delta}^{\varepsilon},\hat{Y}_{s}^{\varepsilon})dW^{2}_s,
\end{align*}
which equals to
\begin{align}
\hat{Y}_{t}^{\varepsilon}=y+\frac{1}{\varepsilon}\int_{0}^{t}
B(s/\vare,X_{s(\delta)}^{\varepsilon},\tilde{\hat{Y}}_{s}^{\varepsilon})ds
+\frac{1}{\sqrt{\varepsilon}}\int_{0}^{t}G_2(s/\vare,X_{s(\delta)}^{\varepsilon},\tilde{\hat{Y}}_{s}^{\varepsilon})dW^{2}_s,\label{4.6a}
\end{align}
where $\tilde{\hat{Y}}^{\varepsilon}$ is any $V_2$-valued progressively measurable $dt\otimes \PP$-version of $\check{\hat{Y}}^{\vare}$.

\vspace{0.1cm}
By the construction of $\hat{Y}_{t}^{\varepsilon}$, we obtain the following estimates.

\begin{lemma} \label{DEY}
For any $T>0$, there exist a constant $C_{T}>0$ and $m\in\NN$ such that for all $\vare\in(0,1)$,
\begin{align}
\sup_{t\in[0,T]}\mathbb{E}\|\hat{Y}_{t}^{\vare}\|^2_{H_2}\leq
C_{T}(1+\|x\|^2_{H_1}+\|y\|^2_{H_2})\label{3.13a}
\end{align}
and
\begin{align}
\mathbb{E}\left(\int_0^{T}\|Y_{t}^{\varepsilon}-\hat{Y}_{t}^{\varepsilon}\|^2_{H_2}dt\right)\leq C_{T}\left(1+\|x\|^m_{H_1}+\|y\|^m_{H_2}\right)\delta^{1/2}. \label{3.14}
\end{align}
\end{lemma}

\subsection{Evolution system of measures}\label{sub3.2}
\vspace{0.1cm}

As we discussed in the subsection \ref{subsection 1.2}, in order to investigate the existence and uniqueness of an evolution system of measures for the frozen SDE \eqref{TDFrozenE}, we need to extend the time domain of the noise $W^2$ to the whole line. To do this, letting $\{W^{2,1}_t\}_{t\geq0}$ be an independent copy of the cylindrical Wiener process $\{W^{2}_t\}_{t\geq0}$. Define
\begin{equation*}
\bar{W}_t^2:=
\begin{cases}
W_t^{2}, & \mbox{if } t\in[0,+\infty), \\
W_{-t}^{2,1}, & \mbox{if } t\in(-\infty,0).
\end{cases}
\end{equation*}
The frozen equation associated to the fast motion for fixed slow component $x\in H_1$ is as follows,
\begin{eqnarray}\label{FEQ1}
dY_{t}=B(t,x,Y_{t})dt+G_2(t,x,Y_t)d\bar{W}_{t}^{2},\quad Y_{s}=y\in H_2,
\quad t\geq s.
\end{eqnarray}
Under the conditions \ref{B1}-\ref{B4}, for any fixed $x\in H_1$ and initial data $y\in H_2$, equation $\eref{FEQ1}$ has a unique variational solution
{$Y_{t}^{s,x,y}$} in the sense of Definition \ref{S.S.}, i.e., for  its $dt\otimes \PP$-equivalence class $\hat{Y}$ we have
$\hat{Y}^{x,y}\in L^{\kappa}({[s,s+T]}\times \Omega, dt\otimes \PP; V_2)\cap L^2({[s,s+T]}\times\Omega, dt\otimes \PP; H_2)$ with $\kappa$ as in \ref{B3}, we have $\PP$-a.s.
\begin{align}
Y^{s,x,y}_{t}=y+\int_{s}^{t}B(r,x,\tilde{Y}^{s,x,y}_{r})dr
+\int_{s}^{t}G_2(r,x,Y^{s,x,y}_{r})d\bar{W}^{2}_r,\label{SFE}
\end{align}
where $\{\tilde{Y}^{s,x,y}_t\}_{t\geq s}$ is any $V_2$-valued progressively measurable $dt\otimes\PP$-version of $\hat{Y}^{x,y}$. By the same arguments as in the proof of Lemma \ref{PMY}, it is easy to prove that
\begin{equation}\label{0224:05}
\sup_{t\geq s}\EE\|Y_{t}^{s,x,y}\|^2_{H_2}\leq C(1+\|x\|^2_{H_1}+\|y\|^2_{H_2}).
\end{equation}

Then we can obtain the following result, since its proof follows the same argument as used in the proof of \cite[Theorems 3.9 and 3.14]{CL2023}; see also \cite{CL2021}, we omit the detailed proof.

\begin{proposition}\label{Ergodicity}
Suppose that conditions \ref{B1}-\ref{B4} hold.
Then there exists $\{Y_t^x\}_{t\in\RR}$ satisfies for any $p\geq1$,
\begin{equation}\label{0224:06}
\sup_{t\in\RR}\EE \|Y_t^x\|_{H_2}^{2p}\leq C(1+\|x\|_{H_1}^{2p}),
\end{equation}
where $C>0$ is a constant, and for any $t\geq s$,
\begin{equation*}
Y_t^x=Y_s^x+\int_s^tB(r,x,\tilde{Y}_r^x)dr
+\int_s^t G_2(r,x,Y_r^x)d \bar{W}_r^2.
\end{equation*}
Moreover, for any $t\geq s$, $x\in H_1$ and $y\in H_2$,
\begin{equation}\label{WErodicity}
\EE\|Y_t^{s,x,y}-Y_t^x\|_{H_2}^2\leq
{\rm e}^{-\gamma(t-s)}\EE\|y-Y_s^x\|_{H_2}^2.
\end{equation}
\end{proposition}

Let $\{P^{x}_{s,t}\}_{t\geq s}$ be the transition semigroup of the Markov process $\{Y_{t}^{s,x,y}\}_{t\geq s}$,
that is, for any bounded measurable function $\varphi$ on $H_2$,
\begin{align*}
P^x_{s,t} \varphi(y)= \EE \left[\varphi\left(Y_{t}^{s,x,y}\right)\right], \quad y \in H_2,\ \ t\geq s.
\end{align*}
Based on Proposition \ref{Ergodicity} and using the argument as used in \cite{DR2006}, we can easily obtain the following result:
\begin{corollary}\label{cor0528}
For any $x\in H_1$, we denote $\mu_t^x:=\PP\circ[Y_t^x]^{-1},t\in\RR$, where
$\{Y_t^x\}_{t\in\RR}$ is as described in Proposition \ref{Ergodicity}.
Then we have
\begin{equation}\label{0317:03}
\sup_{t\in\RR}\int_{H_2}\|y\|_{H_2}^2\mu_t^x(dy)\leq C(1+\|x\|_{H_1}^2),
\end{equation}
and $\{\mu_t^x\}_{t\in\RR}$ is the evolution system of measures for
$\{P^{x}_{s,t}\}_{t\geq s}$, that is for any $t\ge s$, $x\in H_1$
and $\varphi\in C_b(H_2)$,
\begin{equation}\label{ESM}
\int_{H_2}P^x_{s,t} \varphi(y)\,\mu^x_s(dy)=\int_{H_2}\varphi(y)\,\mu^x_t(dy).
\end{equation}
Meanwhile the mapping $\mu_{\cdot}^x:\RR\rightarrow \mathcal P(H_2)$ is unique with the properties
\eqref{0317:03} and \eqref{ESM}. Moreover, for any Lipschitz continuous function $\phi$ on $H_2$,
\begin{equation}\label{WErodicity2}
\left|P^x_{s,t} \phi(y)-\int_{H_2}\phi(z)\,\mu^x_t(dz) \right|
\leq C {\rm Lip}(\phi)(1+\|x\|_{H_1}+\|y\|_{H_2}){\rm e}^{-\frac{\gamma}{2}(t-s)},
\end{equation}
where ${\rm Lip}(\phi):=\sup_{x\neq y}|\phi(x)-\phi(y)|/\|x-y\|_{H_2}$.
\end{corollary}
\begin{remark}\rm
In fact, the evolution system of measures $\{\mu_t^x\}_{t\in\RR}$ is uniformly almost
periodic provided $B$ and $G_2$ are uniformly almost periodic; see Lemma \ref{Le5.2} for more details.
\end{remark}

\section{Proof of main results} \label{Sec4}

In this section, we give the detailed proofs of Theorems \ref{main result 1} and \ref{main result 2}.

\subsection{Proof of first main result}

For any $R>0$, define
\begin{align*}
&\tau^{\varepsilon,1}_R:=\inf\left\{t\geq 0:\int_0^t(1+\|\tilde{X}^{\varepsilon}_s\|_{V_1}^{\alpha})(1+\|X^{\varepsilon}_s\|_{H_1}^{\beta})ds\geq R\right\},\\
&\tau^{\varepsilon,2}_R:=\inf\left\{t\geq 0:\int_0^t(1+\|\tilde{\bar{X}}^{\varepsilon}_s\|_{V_1}^{\alpha})(1+\|\bar{X}^{\varepsilon}_s\|_{H_1}^{\beta})ds\geq R\right\}.
\end{align*}
Let $\tau^{\varepsilon}_R:=\tau^{\varepsilon,1}_R\wedge\tau^{\varepsilon,2}_R$.
We prove that $X_{t}^{\vare}$ strongly converges to $\bar{X}^{\varepsilon}_t$
when time before the stopping time $\tau^{\varepsilon}_R$ firstly,
then the proof of Theorem \ref{main result 1} will follow from the fact that
the difference process $X^{\vare}_t-\bar{X}^{\varepsilon}_t$ after the stopping time
is sufficient small when $R$ is large enough.
We remark that the constants $C$ in the following proofs depend on the initial values
$x$ and $y$ due to repeated use of moment estimates for $X_t^\varepsilon$, $Y_t^\varepsilon$,
$\bar{X}_t^\varepsilon$ and $\hat{Y}_t^\varepsilon$. For simplicity,
we will suppress this dependence in what follows.

We begin by recalling the averaged equation \eqref{AVE1}. Since $F$ satisfies condition \eqref{0525:01}, it is easy to check that the averaged coefficient $\bar{F}$ defined in \eqref{e:drift} is Lipschitz continuous and linear growth, i.e. for all $t\in \RR,u,v\in H_1$,
\begin{equation}\label{0528:03}
\|\bar{F}(t,u)-\bar{F}(t,v)\|_{H_1}\leq C\|u-v\|_{H_1},\quad
\|\bar{F}(t,u)\|_{H_1}\leq C\left(1+\|u\|_{H_1}\right).
\end{equation}
Thus equation $\eref{AVE1}$ has a unique variational solution $\bar X^{\vare}$ in the sense of Definition \ref{S.S.}, i.e., for its $dt\otimes \PP$-equivalence class $\check{\bar{X}}^{\varepsilon}$ we have $\check{\bar{X}}^{\varepsilon}\in L^{\alpha}([0, T]\times \Omega, dt\otimes \PP; V_1)\cap L^2([0, T]\times\Omega, dt\otimes \PP; H_1)$ with $\alpha$ as in \ref{A3}, we have $\PP$-a.s.
\begin{align}\label{4.6b}
\bar{X}^{\varepsilon}_{t}
=x+\int_{0}^{t}A({s/\varepsilon},\tilde{\bar{X}}^{\varepsilon}_{s})ds
+\int_{0}^{t}\bar{F}(s/\varepsilon,\bar{X}^{\varepsilon}_{s})ds
+\int_{0}^{t}G_1({s/\varepsilon},\bar{X}^{\varepsilon}_{s})dW^{1}_s,
\end{align}
where $\tilde{\bar{X}}^{\varepsilon}$ is any $V_1$-valued progressively measurable $dt\otimes \PP$-version of $\check{\bar{X}}^{\varepsilon}$.

Here we present the following estimates. Since their proofs follow essentially the same
arguments as those in Lemmas \ref{PMY} and \ref{COX}, we omit the details here for brevity.
\begin{lemma}\label{L3.8}
For any $T>0$, $p\geq 1$, there exist constants $C_{p,T},C_T>0$ and $m>0$ such that for any $x\in H_1$,
\begin{align}\label{0318:01}
\mathbb{E}\left(\sup_{t\in[0,T]}\|\bar{X}^{\varepsilon}_{t}\|^{2p}_{H_1}\right)+\EE\left(\int_0^T\|\bar{X}^{\varepsilon}_{t}\|^{2p-2}_{H_1}\|\tilde{\bar{X}}^{\varepsilon}_{t}\|_{V_1}^{\alpha}dt\right)\leq C_{p,T}(1+\|x\|^{2p}_{H_1})
\end{align}
and
\begin{align}
\mathbb{E}\left[\int^{T}_0\|\bar{X}^{\varepsilon}_{t}-\bar{X}^{\varepsilon}_{t(\delta)}\|^2_{H_1} dt\right]\leq C_{T}\delta^{1/2}(1+\|x\|^m_{H_1}).\label{barXT}
\end{align}
\end{lemma}

\begin{proposition} \label{ESX}
Assume that conditions \ref{A1}-\ref{A5} and \ref{B1}-\ref{B4} hold. Then
for any $(x,y)\in H_1\times H_2$ and $T,R>0$,
there exist constants $m>0$ and $C_{R,T}>0$ such that for all $\vare\in(0,1)$,
\begin{align}\label{CBS}
\mathbb{E}\left(\sup_{t\in[0, T\wedge {\tau_R^\varepsilon}]}
\|X_{t}^{\vare}-\bar{X}^{\varepsilon}_{t}\|^2_{H_1}\right)\leq C_{R,T}\left(1+\|x\|^m_{H_1}+\|y\|^m_{H_2}\right)
\varepsilon^{1/6}.
\end{align}
\end{proposition}

\begin{proof}
We divide the proof into three steps.

\vspace{2mm}
\textbf{Step 1.} Applying It\^{o}'s formula, \ref{A2} and \eqref{0525:01}, we have
\begin{align}\label{0528:04}
\|X_{t}^{\vare}-\bar{X}^{\varepsilon}_{t}\|^2_{H_1}
=&
2\int_0^t{_{V^{*}_1}}\langle A(s/\varepsilon,\tilde{X}_s^\varepsilon)
-A(s/\varepsilon,\tilde{\bar{X}}^{\varepsilon}_s), \tilde{X}_{s}^{\vare}-\tilde{\bar{X}}^{\varepsilon}_{s}\rangle_{V_1} ds\nonumber\\
&
+\int_0^t\|G_{1}(s/\vare,X_s^\varepsilon)-G_{1}(s/\vare,\bar{X}^{\varepsilon}_s)
\|_{L_{2}(U_1, H_1)}^2ds\nonumber\\
&
+2\int_0^t\left\langle F(s/\vare,X_{s}^\varepsilon,Y_s^\varepsilon)
-\bar{F}(s/\vare,\bar{X}^{\varepsilon}_s), X_{s}^{\vare}-\bar{X}^{\varepsilon}_{s}\right\rangle_{H_1} ds\nonumber\\
&
+2\int_0^t\langle X_{s}^{\vare}-\bar{X}^{\varepsilon}_{s}, [G_{1}(s/\vare,X_s^\varepsilon)-G_{1}(s/\vare,\bar{X}^{\varepsilon}_s)]dW_s^{1}\rangle_{H_1}\nonumber\\
\leq&
C\int_0^t\left(1+\rho\left(\tilde{X}_s^\varepsilon\right)
+\eta\left(\tilde{\bar{X}}^{\varepsilon}_s\right)\right) \left\|X_{s}^{\vare}-\bar{X}^{\varepsilon}_{s}\right\|_{H_1}^2 ds+\sI^{\vare}(t)+\cM^{\vare}(t),
\end{align}
where
$$
\sI^{\vare}(t):=2\int_0^t\left\langle F(s/\vare,X_{s}^\varepsilon,Y_s^\varepsilon)
-\bar{F}(s/\vare,\bar{X}^{\varepsilon}_s), X_{s}^{\vare}-\bar{X}^{\varepsilon}_{s}\right\rangle_{H_1} ds,
$$
$$
\cM^{\vare}(t):=2\int_0^t\langle X_{s}^{\vare}-\bar{X}^{\varepsilon}_{s}, [G_{1}(s/\vare,X_s^\varepsilon)-G_{1}(s/\vare,\bar{X}^{\varepsilon}_s)]dW_s^{1}\rangle_{H_1}.
$$
It follows from \ref{A5}, \eqref{0528:03} and H\"older's inequality that
\begin{align*}
\sI^{\vare}(t)
=&
2\int_0^t\left\langle\bar{F}(s/\vare,X_{s}^\varepsilon)
-\bar{F}(s/\vare,\bar{X}^{\varepsilon}_s), X_{s}^{\vare}-\bar{X}^{\varepsilon}_{s}\right\rangle_{H_1} ds\nonumber\\
&
+2\int_0^t\left[\langle F(s/\vare,X_{s}^\varepsilon,Y_s^\varepsilon)-F(s/\vare,X_{s(\delta)}^\varepsilon,\hat{Y}_s^\varepsilon)
, X_{s}^{\vare}-\bar{X}^{\varepsilon}_{s}\rangle_{H_1}\right.\\
&\quad\quad\quad
+\left.\langle \bar{F}(s/\vare,X^{\vare}_{s(\delta)})-\bar{F}(s/\vare,X^{\vare}_s), X_{s}^{\vare}-\bar{X}^{\varepsilon}_{s}\rangle_{H_1}\right]ds\nonumber\\
&
+2\int_0^t\left\langle F(s/\vare,X_{s(\delta)}^\varepsilon,\hat{Y}_s^\varepsilon)
-\bar{F}(s/\vare,X^{\vare}_{s(\delta)}), X_{s}^{\vare}-X_{s(\delta)}^{\vare}-\bar{X}^{\varepsilon}_{s}+\bar{X}^{\varepsilon}_{s(\delta)}\right\rangle_{H_1} ds\nonumber\\
&
+2\int_0^t\left\langle F(s/\vare,X_{s(\delta)}^\varepsilon,\hat{Y}_s^\varepsilon)-\bar{F}(s/\vare,X^{\vare}_{s(\delta)}), X_{s(\delta)}^{\vare}-\bar{X}^{\varepsilon}_{s(\delta)}\right\rangle_{H_1} ds\nonumber\\
\leq&
C\int_0^t\left(\left\|X_{s}^{\vare}-\bar{X}^{\varepsilon}_{s}\right\|_{H_1}
+\left\|X_{s}^{\vare}-X_{s(\delta)}^{\vare}\right\|_{H_1}
+\left\|Y_{s}^{\vare}-\hat{Y}_{s}^{\vare}\right\|_{H_2}\right)
\left\|X_{s}^{\vare}-\bar{X}^{\varepsilon}_{s}\right\|_{H_1}ds\nonumber\\
&
+C\int_0^t\left(1+\left\|X_{s(\delta)}^{\vare}\right\|_{H_1}
+\left\|\hat{Y}_{s}^{\vare}\right\|_{H_2}\right)
\left(\left\|X_{s}^{\vare}-X_{s(\delta)}^{\vare}\right\|_{H_1}
+\left\|\bar{X}_{s}^{\vare}-\bar{X}_{s(\delta)}^{\vare}\right\|_{H_1}\right)ds
+2\sI^{\vare}_1(t)\nonumber\\
\leq&
C\int_0^t\left(\left\|X_{s}^{\vare}-\bar{X}^{\varepsilon}_{s}\right\|_{H_1}^2
+\left\|X_{s}^{\vare}-X_{s(\delta)}^{\vare}\right\|_{H_1}^2
+\left\|Y_{s}^{\vare}-\hat{Y}_{s}^{\vare}\right\|_{H_2}^2\right)ds
+2\sI^{\vare}_1(t)\nonumber\\
&
+C\left[\int_0^T\left(1+\left\|X_{s(\delta)}^{\vare}\right\|_{H_1}^2
+\left\|\hat{Y}_{s}^{\vare}\right\|_{H_2}^2\right)ds\right]^{\frac{1}{2}}
\left[\int_0^T\left\|X_{s}^{\vare}-X_{s(\delta)}^{\vare}\right\|_{H_1}^2
+\left\|\bar{X}_{s}^{\vare}-\bar{X}_{s(\delta)}^{\vare}\right\|_{H_1}^2ds\right]^{\frac{1}{2}},
\end{align*}
where
$$
\sI^{\vare}_1(t):=\int_0^t\left\langle F(s/\vare,X_{s(\delta)}^\varepsilon,\hat{Y}_s^\varepsilon)
-\bar{F}(s/\vare,X^{\vare}_{s(\delta)}), X_{s(\delta)}^{\vare}-\bar{X}^{\varepsilon}_{s(\delta)}\right\rangle_{H_1} ds.
$$
Then \eqref{F3.1}, \eqref{3.13a}, \eqref{barXT}, \eqref{3.14}, \eqref{F3.7} and
H\"older's inequality yield that for any $t\in[0,T]$,
\begin{align}\label{0528:05}
\EE\left(\sup_{s\in[0, t\wedge{\tau^{\varepsilon}_R}]}\left|\sI^{\vare}(t)\right|\right)
\leq&
 C\int_0^t\EE\left(\sup_{r\in[0, s\wedge{\tau^{\varepsilon}_R}]}
\left\|X_{r}^{\vare}-\bar{X}^{\varepsilon}_{r}\right\|_{H_1}^2\right)ds\\\nonumber
&
+C_{T}(1+\|x\|^m_{H_1}+\|y\|^m_{H_2})\delta^{\frac{1}{4}}
+C\EE\left(\sup_{s\in[0, t\wedge{\tau^{\varepsilon}_R}]}\left|\sI^{\vare}_1(s)\right|\right),
\end{align}
where $m>0$.

Note that $\left\{\cM(t)\right\}_{0\leq s\leq T}$ is a local martingale,
thus applying Burkholder-Davis inequality (see e.g. \cite[Proposition D.0.1]{LR2015}),
Young's inequality and \ref{A5}, we have
\begin{align}\label{0528:06}
\EE\left(\sup_{s\in[0,t\wedge{\tau^{\varepsilon}_R}]}
\left|\cM^{\vare}(s)\right|\right)
\leq
&\frac{1}{2}\mathbb{E}\left(\sup_{s\in[0, t\wedge{\tau^{\varepsilon}_R}]}
\|X_{s}^{\vare}-\bar{X}^{\varepsilon}_{s}\|^2_{H_1}\right)\\\nonumber
&
+C_{R,T}\int_0^{t}
\mathbb{E}\left(\sup_{r\in[0,s\wedge{\tau^{\varepsilon}_R}]}
\|X_{r}^{\vare}-\bar{X}^{\varepsilon}_{r}\|^2_{H_1}\right)ds.
\end{align}
Combining \eqref{0528:04}, \eqref{0528:05}, \eqref{0528:06} and the definition of
$\tau_R^\vare$, we have
\begin{align*}
\EE\left(\sup_{s\in[0, t\wedge{\tau^{\varepsilon}_R}]}
\|X_{s}^{\vare}-\bar{X}^{\varepsilon}_{s}\|^2_{H_1}\right)
\leq&
C_{R,T}\int_0^{t}
\mathbb{E}\left(\sup_{r\in[0,s\wedge{\tau^{\varepsilon}_R}]}
\|X_{r}^{\vare}-\bar{X}^{\varepsilon}_{r}\|^2_{H_1}\right)ds\\\nonumber
&
+C_{T}(1+\|x\|^m_{H_1}+\|y\|^m_{H_2})\delta^{\frac{1}{4}}
+C\EE\left(\sup_{s\in[0, t\wedge{\tau^{\varepsilon}_R}]}\left|\sI^{\vare}_1(t)\right|\right),
\end{align*}
which along with Gronwall's inequality implies that

\begin{align}\label{F4.2}
\mathbb{E}\left(\sup_{t\in[0, T\wedge{\tau^{\varepsilon}_R}]}
\|X_{t}^{\vare}-\bar{X}^{\varepsilon}_{t}\|^2_{H_1}\right)
\leq
C_{R,T}(1+\|x\|^m_{H_1}+\|y\|^m_{H_2})\delta^{\frac{1}{4}}
+C_{R,T}\EE\left(\sup_{t\in[0, T\wedge{\tau^{\varepsilon}_R}]}\left|\sI^{\vare}_1(t)\right|\right).
\end{align}

\vspace{2mm}
\textbf{Step 2.} In this step, we mainly estimate the last term in the above inequality. Note that
\begin{align}\label{J}
\left|\sJ^{\vare}_1(t)\right|
=&
\left|\sum_{k=0}^{[t/\delta]-1}
\int_{k\delta}^{(k+1)\delta}\langle F(s/\vare,X_{s(\delta)}^{\vare},\hat{Y}_{s}^{\vare})-\bar{F}(s/\vare,X^{\vare}_{s(\delta)}), X_{s(\delta)}^{\vare}-\bar{X}^{\varepsilon}_{s(\delta)}\rangle_{H_1} ds\right.\nonumber\\
&
\left.+\int_{t(\delta)}^{t}\langle F(s/\vare,X_{s(\delta)}^{\vare},\hat{Y}_{s}^{\vare})-\bar{F}(s/\vare,X^{\vare}_{s(\delta)}), X_{s(\delta)}^{\vare}-\bar{X}^{\varepsilon}_{s(\delta)}\rangle_{H_1} ds\right|\nonumber\\
\leq&
\sum_{k=0}^{[t/\delta]-1}
\left|\int_{k\delta}^{(k+1)\delta}\langle F(s/\vare,X_{s(\delta)}^{\vare},\hat{Y}_{s}^{\vare})-\bar{F}(s/\vare,X^{\vare}_{s(\delta)}), X_{s(\delta)}^{\vare}-\bar{X}^{\varepsilon}_{s(\delta)}\rangle_{H_1} ds\right|\nonumber\\
&
+\left|\int_{t(\delta)}^{t}\langle F(s/\vare,X_{s(\delta)}^{\vare},\hat{Y}_{s}^{\vare})-\bar{F}(s/\vare,X^{\vare}_{s(\delta)}), X_{s(\delta)}^{\vare}-\bar{X}^{\varepsilon}_{s(\delta)}\rangle_{H_1} ds\right|\nonumber\\
=:&
\sJ^{\vare}_{11}(t)+\sJ^{\vare}_{12}(t).
\end{align}

For the term $\sJ^{\vare}_{12}(t)$, by H\"older's inequality, \ref{A5}, \eqref{F3.1}, \eqref{3.13a} and \eqref{0318:01},
it is easy to see
\begin{align}
\EE\left[\sup_{t\in [0, T]}\sJ^{\vare}_{12}(t)\right]
\leq &
C\left[\EE\sup_{t\in [0, T]}\|X^{\vare}_t-\bar{X}^{\varepsilon}_{t}\|^2_{H_1}\right]^{1/2}
\left[\EE\sup_{t\in[0,T]}\left|\int_{t(\delta)}^{t}(1+\|X^{\vare}_{s(\delta)}\|_{H_1}
+\|\hat{Y}_{s}^{\vare}\|_{H_2})ds\right|^2\right]^{1/2}\nonumber\\
\leq&
C\left[\EE\sup_{t\in [0, T]}\|X^{\vare}_t-\bar{X}^{\varepsilon}_{t}\|^2_{H_1}\right]^{1/2}
\left[\EE\int_{0}^{T}(1+\|X^{\vare}_{s(\delta)}\|^2_{H_1}
+\|\hat{Y}_{s}^{\vare}\|^2_{H_2})ds\right]^{1/2}\delta^{1/2}\nonumber\\
\leq&
C_{T}\left(1+\|x\|^{2}_{H_1}+\|y\|^{2}_{H_2}\right)\delta^{1/2}.\label{J2}
\end{align}

For the term $\sJ^{\vare}_{11}(t)$, employing H\"older's inequality, \eqref{F3.1} and \eqref{0318:01}
we have
\begin{align*}
\mathbb{E}\left[\sup_{t\in[0, T]}\sJ^{\vare}_{11}(t)\right]
\leq&
\sum_{k=0}^{[T/\delta]-1}\mathbb{E}\left|\int_{k\delta}^{(k+1)\delta}
\langle F(s/\vare,X_{k\delta}^{\vare},\hat{Y}_{s}^{\vare})-\bar{F}(s/\vare,X_{k\delta}^{\vare}), X_{k\delta}^{\vare}-\bar{X}^{\varepsilon}_{k\delta}\rangle_{H_1} ds\right|\nonumber\\
\leq&
\sum_{k=0}^{[T/\delta]-1}\left[\mathbb{E}\|X^{\vare}_{k\delta}-\bar X^{\vare}_{k\delta}\|^2_{H_1}\right]^{1/2}\\
&\quad\quad\quad
\times\left[\EE\left\|\int_{k\delta}^{(k+1)\delta}
F(s/\vare,X_{k\delta}^{\vare},\hat{Y}_{s}^{\vare})-\bar{F}(s/\vare,X_{k\delta}^{\vare})ds\right\| ^2_{H_1}\right]^{1/2}\nonumber\\
\leq&
C_{T}\left(1+\|x\|_{H_1}+\|y\|_{H_2}\right)\sum_{k=0}^{[T/\delta]-1}\left[\int_{k\delta}^{(k+1)\delta}
\int_{r}^{(k+1)\delta}\Psi^{\varepsilon}_{k}(s,r)dsdr\right]^{1/2},
\end{align*}
where for any $k\delta\leq r\leq s\leq (k+1)\delta$,
\begin{align*}
\Psi^{\varepsilon}_{k}(s,r):=\mathbb{E}
\langle F(s/\vare,X_{k\delta}^{\vare},\hat{Y}_{s}^{\vare})-\bar{F}(s/\vare,X_{k\delta}^{\vare}),
F(r/\vare,X_{k\delta}^{\vare},\hat{Y}_{r}^{\vare})-\bar{F}(r/\vare,X_{k\delta}^{\vare})\rangle_{H_1},
\end{align*}
and we will give the following estimate in \textbf{Step 3}:
\begin{equation}\label{F4.5}
\Psi^{\varepsilon}_{k}(s,r)\leq
C_{T}\left(1+\|x\|^{2}_{H_1}+\|y\|^{2}_{H_2}\right){\rm e}^{-\frac{\gamma(s-r)}{2\varepsilon}}.
\end{equation}
Then we have
\begin{align}\label{0224:08}
\mathbb{E}\left[\sup_{t\in[0, T]}\sJ^{\vare}_{11}(t)\right]
&
\leq C_T\left(1+\|x\|_{H_1}+\|y\|_{H_2}\right)\sum_{k=0}^{[T/\delta]-1}
\left[\int_{k\delta}^{(k+1)\delta}\int_{r}^{(k+1)\delta}
\Psi^{\varepsilon}_{k}(s,r)dsdr\right]^{1/2}\nonumber\\
&
\leq  C_{T}\delta^{-1}\left(1+\|x\|^2_{H_1}+\|y\|^2_{H_2}\right)
\max_{0\leq k\leq [T/\delta]-1}
\left(\int_{k\delta}^{(k+1)\delta}\int_{r}^{(k+1)\delta}
{\rm e}^{-\frac{\gamma(s-r)}{2\varepsilon}} ds dr\right)^{1/2}\nonumber\\
&\leq C_{T}\left(1+\|x\|^2_{H_1}+\|y\|^2_{H_2}\right)\left(\varepsilon/\delta\right)^{1/2}.
\end{align}
Hence, \eqref{0528:04}, \eqref{J}, \eqref{J2} and \eqref{0224:08} yield that
\begin{align*}
\EE\left(\sup_{t\in[0, T\wedge{\tau^{\varepsilon}_R}]}\left|\sI^{\vare}_1(t)\right|\right)
\leq C_{T}\left(1+\|x\|^2_{H_1}+\|y\|^2_{H_2}\right)
\left[\left(\varepsilon/\delta\right)^{1/2}+\delta^{1/2}\right].
\end{align*}
This together with \eref{F4.2}, and taking $\delta=\varepsilon^\frac{2}{3}$,
it is easy to see \eref{CBS} holds.

\vspace{2mm}
\textbf{Step 3.} In this step, we intend to prove \eref{F4.5}. For any $s>0$, and any $\mathscr{F}_s$-measurable $H_1$-valued random variable $X$ and $H_2$-valued random variable $Y$, we consider the following equation:
\begin{eqnarray*}
\left\{ \begin{aligned}
&dY_{t}=\frac{1}{\vare}B(t,X,Y_{t})dt+\frac{1}{\sqrt{\vare}}G_2(t,X,Y_t)dW_{t}^{2},\quad t\geq s,\\
&Y_{s}=Y,
\end{aligned} \right.
\end{eqnarray*}
which has a unique solution $\tilde{Y}^{\vare,s,X,Y}_t$. Then by the construction of $\hat{Y}_{t}^{\vare}$,
for any $k\in \mathbb{N}$ and $t\in[k\delta,(k+1)\delta]$ we have $\PP$-a.s.,
$$
\hat{Y}_{t}^{\vare}=\tilde Y^{\vare,k\delta,X_{k\delta}^{\vare},\hat{Y}_{k\delta}^{\vare}}_t,
$$
which implies
\begin{align*}
\Psi^{\varepsilon}_{k}(s,r)
=&
\mathbb{E}
\big\langle F(s/\vare,X_{k\delta}^{\vare},\tilde {Y}^{\vare, k\delta,X_{k\delta}^{\vare},
\hat Y_{k\delta}^{\vare}}_{s})
-\bar{F}(s/\vare,X_{k\delta}^{\vare}),
F(r/\vare,X_{k\delta}^{\vare},\tilde{Y}^{\vare, k\delta,X_{k\delta}^{\vare}, \hat Y_{k\delta}^{\vare}}_{r})-\bar{F}(r/\vare,X_{k\delta}^{\vare})\big\rangle_{H_1}\\
=&
\int_{\Omega}\mathbb{E}\Big[
\big\langle F(s/\vare,X_{k\delta}^{\vare},\tilde {Y}^{\vare, k\delta,X_{k\delta}^{\vare}, \hat Y_{k\delta}^{\vare}}_{s})-\bar{F}(s/\vare,X_{k\delta}^{\vare}), \\
&\quad\quad\quad
F(r/\vare,X_{k\delta}^{\vare},\tilde{Y}^{\vare, k\delta,X_{k\delta}^{\vare}, \hat Y_{k\delta}^{\vare}}_{r})-\bar{F}(r/\vare,X_{k\delta}^{\vare})\big\rangle_{H_1}| \mathscr{F}_{k\delta}\Big](\omega)\PP(d \omega)\\
=&
\int_{\Omega}
\mathbb{E}\big\langle F(s/\vare,X_{k\delta}^{\vare}(\omega),\tilde {Y}^{\vare, k\delta,X_{k\delta}^{\vare}(\omega), \hat Y_{k\delta}^{\vare}(\omega)}_{s})-\bar{F}(s/\vare,X_{k\delta}^{\vare}(\omega)), \\
&\quad\quad\quad
F(r/\vare,X_{k\delta}^{\vare}(\omega),\tilde{Y}^{\vare, k\delta,X_{k\delta}^{\vare}(\omega), \hat Y_{k\delta}^{\vare}(\omega)}_{r})-\bar{F}(r/\vare,X_{k\delta}^{\vare}(\omega))\big\rangle_{H_1}\PP(d \omega),
\end{align*}
where the last equality comes from the fact that $X_{k\delta}^{\vare}$ and $\hat Y_{k\delta}^{\vare}$ are $\mathscr{F}_{k\delta}$-measurable, and for any fixed $(x,y)\in H_1\times H_2$, $\{\tilde Y^{\vare, k\delta,x,y}_{s}\}_{s\geq k\delta}$ is independent of $\mathscr{F}_{k\delta}$.

By the definition of process $\tilde{Y}^{\vare,k\delta,x,y}_t$, for its $dt\otimes \PP$-equivalence class $\check{\tilde{Y}}^{\vare,k\delta,x,y}$ we have $\check{\tilde{Y}}^{\vare,k\delta,x,y}\in L^{\kappa}([k\delta, T]\times \Omega, dt\otimes \PP; V_2)\cap L^2([k\delta, T]\times\Omega, dt\otimes \PP; H_2)$ with $\kappa$ as in \ref{B3} and $\PP$-a.s.
\begin{align}\label{E3.12.1}
\tilde{Y}^{\vare,k\delta,x,y}_{s\vare}=
&
y+\frac{1}{\vare}\int^{s\vare}_{k\delta} B(r/\vare,x,\tilde{\tilde{Y}}^{\vare,k\delta,x,y}_r)dr+\frac{1}{\sqrt{\vare}}\int^{s\vare}_{k\delta} G_2(r/\vare,x,\tilde{\tilde{Y}}^{\vare,k\delta,x,y}_r)dW^2_r\\
=&
y+\int^{s}_{k\delta/\vare} B(r,x,\tilde{\tilde{Y}}^{\vare,k\delta,x,y}_{r\vare})dr+\int^{s}_{k\delta/\vare} G_2(r,x,\tilde{Y}^{\vare,k\delta,x,y}_{r\vare})d\hat{W}^{2}_r,\quad s\in [k\delta/\vare,(k+1)\delta/\vare],\nonumber
\end{align}
where $\tilde{\tilde{Y}}^{\vare,k\delta,x,y}$ is any $V_2$-valued progressively measurable
$dt\otimes \PP$-version of $\check{\tilde{Y}}^{\vare,k\delta,x,y}$, and
$\{\hat W^{2}_t:=\frac{1}{\sqrt{\vare}}W^2_{t\vare}
\}_{t\in[k\delta/\varepsilon,(k+1)\delta/\varepsilon]}$.

The uniqueness of the solution of equations (\ref{E3.12.1}) and (\ref{SFE}) implies
that the distribution of $\{\tilde Y^{\vare, k\delta, x,y}_{s\vare}\}_{s\in [k\delta/\vare,(k+1)\delta/\vare]}$
coincides with the distribution of $\{Y_{s}^{k\delta/\vare,x, y}\}_{s\in [k\delta/\vare,(k+1)\delta/\vare]}$.
Then by \ref{A5}, \eref{WErodicity2} and Markov property, we have
\begin{align}\label{0224:07}
&
\mathbb{E}\big\langle F(s/\vare,X_{k\delta}^{\vare}(\omega),\tilde {Y}^{\vare, k\delta,X_{k\delta}^{\vare}(\omega), \hat Y_{k\delta}^{\vare}(\omega)}_{s})
-\bar{F}(s/\vare,X_{k\delta}^{\vare}(\omega)), \nonumber\\
&\quad\quad\quad
F(r/\vare,X_{k\delta}^{\vare}(\omega),\tilde{Y}^{\vare, k\delta,X_{k\delta}^{\vare}(\omega), \hat Y_{k\delta}^{\vare}(\omega)}_{r})
-\bar{F}(r/\vare,X_{k\delta}^{\vare}(\omega))\big\rangle_{H_1}\nonumber\\
&=
\EE\big\langle F(s/\vare,X_{k\delta}^{\vare}(\omega),Y_{s/\varepsilon}^{k\delta/\varepsilon,X_{k\delta}^\varepsilon(\omega),
\hat{Y}_{k\delta}^\varepsilon(\omega)})-\bar{F}(s/\vare,X_{k\delta}^{\vare}(\omega)), \nonumber\\
&\quad\quad\quad
F(r/\vare,X_{k\delta}^{\vare}(\omega),Y_{r/\varepsilon}^{k\delta/\varepsilon,X_{k\delta}^\varepsilon(\omega),
\hat{Y}_{k\delta}^\varepsilon(\omega)})-\bar{F}(r/\vare,X_{k\delta}^{\vare}(\omega))\big\rangle_{H_1}\nonumber\\
&
=\int_{\Omega}\big\langle \EE \Big[F(s/\vare,X_{k\delta}^{\vare}(\omega),
Y^{r/\vare,X_{k\delta}^{\vare}(\omega),Y_{r/\vare}^{k\delta/\vare,X_{k\delta}^{\vare}(\omega),
\hat Y_{k\delta}^{\vare}(\omega)}(\tilde\omega)}_{s/\vare})
-\bar{F}(s/\vare, X_{k\delta}^{\vare}(\omega))\Big],\nonumber\\
&\quad\quad\quad
F(r/\varepsilon,X_{k\delta}^\varepsilon(\omega),Y_{r/\varepsilon}^{k\delta/\varepsilon,X_{k\delta}^\varepsilon(\omega),
\hat{Y}_{k\delta}^\varepsilon(\omega)}(\tilde\omega))
-\bar{F}(r/\vare,X_{k\delta}^\varepsilon(\omega))
\big\rangle_{H_1}\PP(d\tilde \omega)\nonumber\\
&
\leq C\int_{\Omega}\left(1+\|X_{k\delta}^\vare(\omega)\|_{H_1}^2
+\left\|Y_{r/\vare}^{k\delta/\vare,X_{k\delta}^{\vare}(\omega),
\hat Y_{k\delta}^{\vare}(\omega)}(\tilde\omega)\right\|_{H_2}^2\right)
{\rm e}^{-\frac{\gamma(s-r)}{2\varepsilon}}\PP(d\tilde\omega),
\end{align}
which along with \eref{0224:05}, \eref{F3.1} and \eref{3.13a} implies that
\begin{align*}
\Psi_{k}^\vare(s,r)
\leq&
C_T\int_{\Omega}\left[(1+\|X^{\vare}_{k\delta}(\omega)\|^{2}_{H_1}+\|\hat Y_{k\delta}^{\vare}(\omega)\|^{2}_{H_2})\right]\PP(d\omega){\rm e}^{-\frac{\gamma(s-r)}{2\varepsilon}}\\
\leq&
C_{T}\left(1+\|x\|^{2}_{H_1}+\|y\|^{2}_{H_2}\right){\rm e}^{-\frac{\gamma(s-r)}{2\varepsilon}},
\end{align*}
which gives  estimate \eref{F4.5}. The proof is complete.
\end{proof}

\vspace{0.1cm}
Now, we are in a position to prove our first main result.

\noindent\textbf{Proof of Theorem \ref{main result 1}:}
Applying Chebyshev's inequality, Lemmas \ref{PMY} and \ref{L3.8}, we have
\begin{align}
&
\mathbb{E}\left(\sup_{t\in [0, T]}\|X_{t}^{\vare}-\bar{X}^{\varepsilon}_{t}\|^2_{H_1} 1_{\{T>\tau_{R}^\vare\}}\right)\nonumber\\
&\leq
\left[\mathbb{E}\left(\sup_{t\in [0, T]}
\|X_{t}^{\vare}-\bar{X}^{\varepsilon}_{t}\|^{4}_{H_1}\right)\right]^{1/2}
\cdot\left[\mathbb{P}\left(T>\tau_{R}^\vare\right)\right]^{1/2} \nonumber\\
&\leq
\frac{C_{T}\left(1+\|x\|^{2}_{H_1}+\|y\|^{2}_{H_2}\right)}{\sqrt{R}}
\left[\mathbb{E}\int_0^T(1+\|\tilde{X}^{\varepsilon}_s\|_{V_1}^{\alpha})
(1+\|X^{\varepsilon}_s\|_{H_1}^{\beta})ds\right]^{1/2}\nonumber\\
&\quad+
\frac{C_{T}\left(1+\|x\|^{2}_{H_1}+\|y\|^{2}_{H_2}\right)}{\sqrt{R}}
\left[\mathbb{E}\int_0^T(1+\|\tilde{\bar{X}}^{\varepsilon}_s\|_{V_1}^{\alpha})
(1+\|\bar{X}^{\varepsilon}_s\|_{H_1}^{\beta})ds\right]^{1/2}\nonumber\\
&\leq
\frac{C_{T}\left(1+\|x\|^{m}_{H_1}+\|y\|^{m}_{H_2}\right)}{\sqrt{R}}, \label{CAS}
\end{align}
for some $m>0$. Then estimates \eref{CBS} and \eref{CAS} yield
\begin{align*}
&
\mathbb{E}\left(\sup_{t\in [0, T]}
\|X_{t}^{\vare}-\bar{X}^{\varepsilon}_{t}\|^2_{H_1}\right)\nonumber\\
&
\leq\mathbb{E}\left(\sup_{t\in [0, T]}\|X^{\vare}_{t}-\bar{X}^{\varepsilon}_{t}\|^2_{H_1}
1_{\{T\leq \tau_{R}^\vare\}}\right)
+\mathbb{E}\left(\sup_{t\in [0, T]}\|X_{t}^{\vare}-\bar{X}^{\varepsilon}_{t}\|^2_{H_1} 1_{\{T>\tau_{R}^\varepsilon\}}\right)\nonumber\\
&
\leq C_{R,T}\left(1+\|x\|^m_{H_1}+\|y\|^m_{H_2}\right)\varepsilon^{\frac{1}{6}}
+\frac{C_{T}\left(1+\|x\|^m_{H_1}+\|y\|^m_{H_2}\right)}{\sqrt{R}}.
\end{align*}
Now, letting $\vare\rightarrow 0$ first, then taking $R\rightarrow \infty$, we have
\begin{align}
\lim_{\vare\rightarrow 0}\mathbb{E}\left(\sup_{t\in [0, T]}\|X_{t}^{\vare}-\bar{X}^{\varepsilon}_{t}\|^2_{H_1}\right)=0.\label{L2}
\end{align}

Note that for any $p>1$, by Lemmas \ref{PMY} and \ref{L3.8}, it is easy to see that
\begin{align*}
\EE\left(\sup_{t\in [0, T]}\|X_{t}^{\vare}-\bar{X}^{\varepsilon}_{t}\|^{2p}_{H_1}\right)
\leq&
\left[\EE\left(\sup_{t\in [0, T]}
\|X_{t}^{\vare}-\bar{X}^{\varepsilon}_{t}\|^{4p-2}_{H_1}\right)\right]^{1/2}\left[\EE\left(\sup_{t\in [0, T]}
\|X_{t}^{\vare}-\bar{X}^{\varepsilon}_{t}\|^{2}_{H_1}\right)\right]^{1/2}\\
\leq&
C_{p, T}\left(1+\|x\|^{2p-1}_{H_1}+\|y\|^{2p-1}_{H_2}\right)\left[\EE\left(\sup_{t\in [0, T]}
\|X_{t}^{\vare}-\bar{X}^{\varepsilon}_{t}\|^{2}_{H_1}\right)\right]^{1/2},
\end{align*}
which along with \eref{L2} completes the proof.
\hspace{\fill}$\Box$

\subsection{Proof of second main result}

Denote by $\cP(H_2)$ the space of all Borel probability measures on $H_2$.
Then $\cP(H_2)$ is a Polish space with the following distance
$$
d_{BL}(\mu,\nu):=\sup\left\{\left|\int_{H_2}\varphi(x_2)\mu(dx_2)-\int_{H_2}\varphi(x_2)\nu(dx_2)\right|:
~\|\varphi\|_{BL}\leq1\right\},\quad \mu,\nu\in\cP(H_2),
$$
where $\|\varphi\|_{BL}:=Lip(\varphi)+|\varphi|_\infty$.
Now we show that the mapping
$$
\RR\times H_1\longrightarrow \cP(H_2),\quad (t,x)\longmapsto \mu_t^{x}
$$
is uniformly almost periodic in the following lemma,
which is critical to construct the averaged coefficient $\bar{F}(\cdot)$.
\begin{lemma}\label{Le5.2}
Suppose that conditions \ref{B1}-\ref{B4} and \ref{C1}-\ref{C3} hold.
Then the evolution system of measures $\{\mu^x_t\}_{t\in\RR}$ given in Proposition $\ref{Ergodicity}$ is uniformly almost periodic.
\end{lemma}
\begin{proof}
It follows from \cite[Corollary 3.15]{CL2023} that for any $x\in H_1$,
$\mu_{\cdot}^x:\RR \to \mathcal P(H_2)$ is almost periodic. Then for any $\{t_n'\}\subset\RR$, there
exists a subsequence $\{t_n\}\subset\{t_n'\}$ and $\widetilde{\mu}^x_t$ such that
\begin{equation*}
\lim_{n\rightarrow\infty}\sup_{t\in\RR}d_{BL}\left(\mu_{t+t_n}^x,\widetilde{\mu}_t^x\right)=0.
\end{equation*}
Now we show that $\{\mu^x_t\}_{t\in\RR}$ is uniformly almost periodic. Specifically, we prove that
\begin{align}
\lim_{n\to\infty}\sup_{t\in \RR,\|x\|_{H_1}\leq R} d_{BL}\left(\mu_{t+t_n}^x,\widetilde{\mu}_t^x\right)=0,\quad \forall R>0.\label{F4.20}
\end{align}

To do this, define
$$
\mu_{n,t}^{x}:=\mu_{t+t_n}^x,\quad
B_{2,n}(t,x,y):=B_2(t+t_n,x,y),\quad G_{2,n}(t,x,y):=G_2(t+t_n,x,y).
$$
Actually, $\mu_{n,t}^{x}=\PP\circ[Y_{n,t}^x]^{-1},t\in\RR$,
where $\{Y_{n,t}^x\}_{t\in\RR}$ satisfies
\begin{equation*}
Y_{n,t}^{x}=Y_{n,s}^{x}+\int_s^tB_n(r,x,\tilde{Y}_{n,r}^{x})dr+\int_s^tG_{2,n}(r,x,Y_{n,r}^{x})d\bar{W}_r^2,
~t\geq s.
\end{equation*}
Here $B_n(r,x,y):=B_1(x,y)+B_{2,n}(r,x,y)$.
Since $B_2$ and $G_2$ are uniformly almost periodic, there exist $B_2': \RR\times H_1\times H_2\rightarrow H_2$, $G_2': \RR\times H_1\times H_2\rightarrow L_2(U_2,H_2)$ and
a subsequence of $\{t_n\}$, which we still denote by $\{t_n\}$, such that for any $M>0$,
\begin{equation}\label{0224:02}
\lim_{n\rightarrow\infty}\sup_{t\in\RR,(x,y)\in \textbf{B}_M}
\|B_2(t+t_n,x,y)-B_2'(t,x,y)\|_{H_2}=0
\end{equation}
and
\begin{equation}\label{0224:03}
\lim_{n\rightarrow\infty}\sup_{t\in\RR,(x,y)\in \textbf{B}_M}\|G_2(t+t_n,x,y)-G_2'(t,x,y)\|_{L_2(U_2,H_2)}=0,
\end{equation}
where
$$
\textbf{B}_{M}:=\{(x,y)\in H_1\times H_2:~\|x\|_{H_1}\vee\|y\|_{H_2}\leq M\}.
$$
Similar to \cite[Proposition 2]{CL2021}, we obtain that
$\{\widetilde{\mu}_t^x\}_{t\in\RR}=
\{\PP\circ[\widetilde{Y}_t^x]^{-1}\}_{t\in\mathbb{R}}$, where $\{\widetilde{Y}_t^x\}_{t\in\RR}$
satisfies
\begin{equation*}
\widetilde{Y}_{t}^{x}=\widetilde{Y}_{s}^{x}
+\int_s^tB'(r,x,\tilde{\widetilde{Y}}_{r}^{x})dr
+\int_s^tG_2'(r,x,\widetilde{Y}_{r}^{x})d\bar{W}_r^2,
\end{equation*}
where $B'(r,x,y):=B_1(x,y)+B_2'(r,x,y)$ and $\tilde{\widetilde{Y}}_{r}^{x}$ is defined similarly to $\tilde{Y}^{s,x,y}_{r}$, introduced below \eqref{SFE}.
In particular,
there exists $C>0$ such that for any $p\geq1$
and $x\in H_1$,
\begin{equation}\label{0220:07}
\EE\|\widetilde{Y}_t^x\|_{H_2}^{2p}=
\int_{H_2}\|y\|_{H_2}^{2p}\widetilde{\mu}_t^x(dy)\leq C(1+\|x\|_{H_1}^{2p}).
\end{equation}

By \ref{B2} and Young's inequality, one sees that for any $-N<t$,
\begin{align*}
&
\EE\|Y_{n,t}^{x}-\widetilde{Y}_{t}^{x}\|_{H_2}^2
-\EE\|Y_{n,-N}^{x}-\widetilde{Y}_{-N}^{x}\|_{H_2}^2\\
&
=2\EE\int_{-N}^t~_{V_2^*}
\langle B_n(s,x,\tilde{Y}_{n,s}^{x})-B'(s,x,\tilde{\widetilde{Y}}_{s}^{x}), \tilde{Y}_{n,s}^{x}-\tilde{\widetilde{Y}}_{s}^{x}\rangle_{V_2} ds\\
&\quad
+\EE\int_{-N}^t\|G_{2,n}(s,x,Y_{n,s}^{x})
-G_2'(s,x,\widetilde{Y}_{s}^{x})\|_{L_2(U_2,H_2)}^2ds\\
&
\leq \EE\int_{-N}^t\bigg(-\gamma\|Y_{n,s}^{x}-\widetilde{Y}_{s}^{x}\|_{H_2}^2
+2\|B_{2,n}(s,x,\widetilde{Y}_{s}^{x})
-B_2'(s,x,\widetilde{Y}_{s}^{x})\|_{H_2}
\|Y_{n,s}^{x}-\widetilde{Y}_{s}^{x}\|_{H_2}\\
&\qquad
+\|G_{2,n}(s,x,\widetilde{Y}_{s}^{x})
-G_2'(s,x,\widetilde{Y}_{s}^{x})\|_{L_2(U_2,H_2)}^2
\\
&\qquad
+2\|G_{2,n}(s,x,\widetilde{Y}_{s}^{x})
-G_2'(t,x,\widetilde{Y}_{s}^{x})\|_{L_2(U_2,H_2)}
\|G_{2,n}(s,x,Y_{n,s}^{x})-G_{2,n}(s,x,\widetilde{Y}_{s}^{x})\|_{L_2(U_2,H_2)}
\bigg)ds\\
&
\leq \EE\int_{-N}^t\left(-\frac{\gamma}{2}
\|Y_{n,s}^{x}-\widetilde{Y}_{s}^{x}\|_{H_2}^2
+C\Lambda_n(s,x)\right)ds,
\end{align*}
which along with Gronwall's inequality, \eqref{0220:07} and \eqref{0224:06} implies that
\begin{align}\label{0509:01}
\EE\|Y_{n,t}^{x}-\widetilde{Y}_{t}^{x}\|_{H_2}^2
&
\leq \EE\|Y_{n,-N}^{x}-\widetilde{Y}_{-N}^{x}\|_{H_2}^2
{\rm e}^{-\frac{\gamma}{2}(t+N)}
+C\int_{-N}^t{\rm e}^{-\frac{\gamma}{2}(t-s)}\EE\Lambda_n(s,x)ds\nonumber\\
&
\leq C(1+\|x\|_{H_1}^2)
{\rm e}^{-\frac{\gamma}{2}(t+N)}
+C\int_{-N}^t{\rm e}^{-\frac{\gamma}{2}(t-s)}\EE\Lambda_n(s,x)ds,
\end{align}
where
\begin{align*}
\Lambda_n(s,x):=
&
\|B_{2,n}(s,x,\widetilde{Y}_{s}^{x})
-B_2'(s,x,\widetilde{Y}_{s}^{x})\|_{H_2}^2\nonumber\\
&
+\left(1+\|x\|_{H_1}+\|\widetilde{Y}_{s}^{x}\|_{H_2}^\zeta
+\|Y_{n,s}^{x}\|_{H_2}^\zeta\right)
\|G_{2,n}(s,x,\widetilde{Y}_{s}^{x})
-G_2'(s,x,\widetilde{Y}_{s}^{x})\|_{L_2(U_2,H_2)}.
\end{align*}

It follows from \ref{C3}, \ref{B4}, \eqref{0224:06} and \eqref{0220:07} that
for any $s\in\RR$ and $x\in H_1$,
\begin{equation}\label{0514:01}
\EE\Lambda^2_n(s,x)
\leq C\left(1+\|x\|_{H_1}^4+\EE\|\widetilde{Y}_s^x\|_{H_2}^4
+\EE\|Y_{n,s}^x\|_{H_2}^4\right)
\leq C\left(1+\|x\|_{H_1}^4\right).
\end{equation}
Letting $N\rightarrow+\infty$ in \eqref{0509:01}, by \eqref{0514:01},
we have for any $l>0$,
\begin{align}\label{0509:04}
\EE\|Y_{n,t}^{x}-\widetilde{Y}_{t}^{x}\|_{H_2}^2
&
\leq \int_{-\infty}^t{\rm e}^{-\frac{\gamma}{2}(t-s)}
\EE\Lambda_n(s,x)ds\nonumber\\
&
=\int_{-\infty}^0{\rm e}^{\frac{\gamma}{2}\tau}\EE\Lambda_n(\tau+t,x)d\tau\nonumber\\
&
=\int_{-\infty}^0{\rm e}^{\frac{\gamma}{2}\tau}\left(\frac{d}{d\tau}\int_t^{\tau+t}
\EE\Lambda_n(r,x)dr\right)d\tau\nonumber\\
&
=\int_{-\infty}^0\frac{\gamma}{2}{\rm e}^{\frac{\gamma}{2}\tau}\int_{\tau+t}^t\EE\Lambda_n(r,x)drd\tau\nonumber\\
&
=\int_{-\infty}^{-l}\frac{\gamma}{2}{\rm e}^{\frac{\gamma}{2}\tau}\int_{\tau+t}^t\EE\Lambda_n(r,x)drd\tau
+\int_{-l}^0\frac{\gamma}{2}{\rm e}^{\frac{\gamma}{2}\tau}\int_{\tau+t}^t\EE\Lambda_n(r,x)drd\tau\nonumber\\
&
\leq C(1+\|x\|_{H_1}^2)(l+\gamma/2){\rm e}^{-\frac{\gamma}{2} l}+ \int_{t-l}^t\EE\Lambda_n(r,x)dr.
\end{align}
Note that for any $M>R>0$ and $\|x\|_{H_1}\leq R$,
\begin{align}\label{0509:02}
\EE\left(\Lambda_n(r,x)1_{\{\|\widetilde{Y}_r^x\|_{H_2}\leq M\}}\right)
&
\leq C\sup_{t\in\RR,(x,y)\in \textbf{B}_M}\|B_{2,n}(t,x,y)-B_2'(t,x,y)\|^2\nonumber\\
&\quad
+C\left(1+R\right)
\sup_{t\in\RR,(x,y)\in \textbf{B}_M}\|G_{2,n}(t,x,y)-G_2'(t,x,y)\|_{L_2(U_2,H_2)}.
\end{align}
By H\"older's inequality, Chebyshev's inequality, \eqref{0514:01}
and \eqref{0220:07}, one sees that for any $\|x\|_{H_1}\leq R$,
\begin{align}\label{0509:03}
\EE\left(\Lambda_n(r,x)1_{\{\|\widetilde{Y}_r^x\|_{H_2}> M\}}\right)
&
\leq \left(\EE\Lambda^2_n(r,x)\right)^{1/2}
\left(\EE1_{\{\|\widetilde{Y}_r^x\|_{H_2}> M\}}\right)^{\frac{1}{2}}\nonumber\\
&
\leq \left(1+\|x\|^{2}_{H_1}\right)
\left(\EE\|\widetilde{Y}_r^x\|_{H_2}^{2}\right)^{\frac{1}{2}}M^{-1}\nonumber\\
&
\leq C(1+R^{2})M^{-1}.
\end{align}

Combining \eqref{0509:04}, \eqref{0509:02} and \eqref{0509:03},
we obtain that for any $R>0$, $M>R$ and $l>0$,
\begin{align}\label{0509:05}
\sup_{t\in\RR,\|x\|_{H_1}\leq R}\EE\|Y_{n,t}^{x}-\widetilde{Y}_{t}^{x}\|_{H_2}^2
&
\leq C(1+R^2)(l+\gamma/2){\rm e}^{-\frac{\gamma}{2} l}+C l(1+R^{2})M^{-1}\\
&\quad
+ lC\sup_{t\in\RR,(x,y)\in \textbf{B}_M}\|B_{2,n}(t,x,y)-B_2'(t,x,y)\|^2\nonumber\\
&\quad
+Cl\left(1+R\right)
\sup_{t\in\RR,(x,y)\in \textbf{B}_M}\|G_{2,n}(t,x,y)-G_2'(t,x,y)\|_{L_2(U_2,H_2)}.\nonumber
\end{align}
Hence, by \eqref{0224:02}, \eqref{0224:03} and \eqref{0509:05}, we have
\begin{align*}
\lim_{n\rightarrow\infty}\sup_{t\in\RR,\|x\|_{H_1}\leq R}
\EE\|Y_{n,t}^{x}-\widetilde{Y}_{t}^{x}\|_{H_2}^2
\leq C(1+R^2)(l+\gamma/2){\rm e}^{-\frac{\gamma}{2} l}+Cl(1+R^{2})M^{-1},
\end{align*}
which implies \eqref{F4.20} holds by first letting $M\rightarrow+\infty$
and then letting $l\rightarrow+\infty$.

\end{proof}

\vspace{0.2cm}

Let $(\mathcal X_{1,i},\|\cdot\|_{\mathcal X_{1,i}})$ be Hilbert space and $(\mathcal X_{2,i},d_{2,i})$ be complete metric space for $i=1,2$.
For any $\varphi\in C(\RR\times\mathcal X_{1,1},\mathcal X_{2,1})$, we denote by $\mathfrak{M}_\varphi$ the family of all sequences
$\{t_n\}\subset \RR$ such that for any $l>0$ and $R>0$,
$\varphi(t_n+t,x)$ converges uniformly in $t\in[-l,l]$
and $\|x\|_{\mathcal X_{1,1}}\leq R$ as $n\rightarrow\infty$. Now we introduce the following result, which is from \cite{Shch};
see also Theorem 2.26 and Remark 2.30 in \cite{CL2020}.

\begin{lemma}\label{RComlem}
Let $\varphi\in C(\RR\times\mathcal X_{1,1},\mathcal X_{2,1})$ and $\psi\in C(\RR\times\mathcal X_{2,1},\mathcal X_{2,2})$.
If $\mathfrak{M}_\psi\subset\mathfrak{M}_\varphi$ and $\psi$ is uniformly almost periodic
then so is $\varphi$.
\end{lemma}

\begin{lemma}\label{FAP}
Suppose that assumptions \ref{A5}, \ref{B1}-\ref{B4} and \ref{C1}-\ref{C3} hold.
Then the averaged coefficient $\bar{F}(t,x)$ defined in $\eqref{e:drift}$ is uniformly almost periodic.
\end{lemma}
\begin{proof}
Denote $(F,\mu):\RR\times H_1\times H_2\to H_1\times \mathcal{P}(H_2)$. Let $\mathfrak{M}_{(F,\mu)}$ be the space of all sequence $\{t_n\}\subset\RR$ such that
there exists $F':\RR\times H_1\times H_2\rightarrow H_1$ and $\widetilde{\mu}:\RR\times H_1\rightarrow \mathcal P(H_2)$ so that for any $R>0$ and $l>0$
\begin{equation}\label{0220:01}
\lim_{n\rightarrow\infty}\sup_{|t|\leq l,(x,y)\in \textbf{B}_R}\|F(t+t_n,x,y)-F'(t,x,y)\|_{H_1}=0
\end{equation}
and
\begin{equation}\label{0220:06}
\lim_{n\rightarrow\infty}\sup_{|t|\leq l,x\in \textbf{B}_R}d_{BL}(\mu_{t+t_n}^x,\widetilde{\mu}_t^x)=0.
\end{equation}
Note that $(F,\mu)$ is uniformly almost periodic by assumption \ref{C3} and Lemma \ref{Le5.2}, thus in order to prove that $\bar{F}(t,x)$ is uniformly almost periodic, it is sufficient to show that $\mathfrak{M}_{(F,\mu)}\subset\mathfrak{M}_{\bar{F}}$ by Lemma \ref{FAP}.

Indeed, taking any $\{t_n\}\in\mathfrak{M}_{(F,\mu)}$, for any $(t,x)\in[-l,l]\times H_1$,
\begin{align}\label{0220:02}
&
\left\|\bar{F}(t+t_n,x)-\int_{H_2}F'(t,x,y)\widetilde{\mu}_t^x(dy)\right\|_{H_1}\nonumber\\
&
=\left\|\int_{H_2}F(t+t_n,x,y)\mu_{t+t_n}^x(dy)
-\int_{H_2}F'(t,x,y)\widetilde{\mu}_t^x(dy)\right\|_{H_1}\nonumber\\
&
\leq\left\|\int_{H_2}\left(F(t+t_n,x,y)-F'(t,x,y)\right)\mu_{t+t_n}^x(dy)\right\|_{H_1}\nonumber\\
&\quad
+\left\|\int_{H_2}F'(t,x,y)\left(\mu_{t+t_n}^x(dy)-\widetilde{\mu}_t^x(dy)\right)\right\|_{H_1}
=:\mathcal I_1+\mathcal I_2.
\end{align}

For $\mathcal I_1$, by \ref{A5} and \eqref{0224:06}, we have for any $R'>0$ and $|t|\leq l$,
\begin{align}\label{0220:03}
\mathcal I_1
&
\leq\left\|\int_{\{\|y\|_{H_2}\leq R'\}}
\left(F(t+t_n,x,y)-F'(t,x,y)\right)\mu_{t+t_n}^x(dy)\right\|\nonumber\\
&\quad
+\left\|\int_{\{\|y\|_{H_2}> R'\}}
\left(F(t+t_n,x,y)-F'(t,x,y)\right)\mu_{t+t_n}^x(dy)\right\|\nonumber\\
&
\leq\sup_{|t|\leq l,\|y\|_{H_2}\leq R'}\|F(t+t_n,x,y)-F'(t,x,y)\|_{H_1}\nonumber\\
&\quad
+C\int_{\{\|y\|_{H_2}>R'\}}\left(1+\|x\|_{H_1}+\|y\|_{H_2}\right)\mu_{t+t_n}^x(dy)\nonumber\\
&
\leq \sup_{|t|\leq l,\|y\|_{H_2}\leq R'}\|F(t+t_n,x,y)-F'(t,x,y)\|_{H_1}
+\frac{C(1+\|x\|_{H_1}^2)}{R'}.
\end{align}
For any $R'>0$, define
\begin{equation*}
F'_{R'}(t,x,y):=
\begin{cases}
F'(t,x,y), & \|y\|_{H_2}\leq R', \\
F'(t,x,y)\left(2-\frac{\|y\|_{H_2}}{R'}\right), & R'\leq \|y\|_{H_2}\leq 2R', \\
0, & \|y\|_{H_2}> 2R'.
\end{cases}
\end{equation*}
Then we have
\begin{align}\label{0220:04}
\mathcal I_2
&
:=\left\|\int_{H_2}F'(t,x,y)\left(\mu_{t+t_n}^x(dy)-\widetilde{\mu}_t^x(dy)\right)\right\|_{H_1}\nonumber\\
&
\leq \left\|\int_{H_2}\left(F'(t,x,y)-F'_{R'}(t,x,y)\right)
\left(\mu_{t+t_n}^x(dy)-\widetilde{\mu}_t^x(dy)\right)\right\|_{H_1}\nonumber\\
&\quad
+\left\|\int_{H_2}F'_{R'}(t,x,y)\left(\mu_{t+t_n}^x(dy)-\widetilde{\mu}_t^x(dy)\right)\right\|_{H_1}\nonumber\\
&
\leq \mathcal I_2^1+\|F'_{R'}(t,x,\cdot)\|_{BL}d_{BL}(\mu_{t+t_n}^x,\widetilde{\mu}_t^x).
\end{align}

Since $B_2$ and $G_2$ are uniformly almost periodic, there exists a subsequence $\{t_n'\}\subset\{t_n\}$,
$B'$ and $G_2'$ such that for any $R>0$,
\begin{equation*}
\lim_{n\rightarrow\infty}\sup_{t\in\RR,(x,y)\in \textbf{B}_R}\|B_2(t+t_n',x,y)-B_2'(t,x,y)\|_{H_2}=0
\end{equation*}
and
\begin{equation*}
\lim_{n\rightarrow\infty}\sup_{t\in\RR,(x,y)\in \textbf{B}_R}\|G_2(t+t_n',x,y)-G_2'(t,x,y)\|_{L_2(U_2,H_2)}=0.
\end{equation*}
Thanks to \cite[Proposition 2]{CL2021}, we obtain that $\{\widetilde{\mu}_t^x\}_{t\in\RR}$ is the evolution system of
measures of
\begin{equation*}
dY_t^x=B'(t,x,Y_t^x)dt+G_2'(t,x,Y_t^x)d\bar W_t^2,
\end{equation*}
where $B'(t,x,y):=B_1(x,y)+B_2'(t,x,y)$.
In view of \ref{A5}, \eqref{0224:06} and \eqref{0220:07}, we get
\begin{align}\label{0220:05}
\mathcal I_2^1
&
:=\left\|\int_{H_2}\left(F'(t,x,y)-F'_{R'}(t,x,y)\right)
\left(\mu_{t+t_n}^x(dy)-\widetilde{\mu}_t^x(dy)\right)\right\|_{H_1}\nonumber\\
&
\leq\int_{\{\|y\|_{H_2}>R'\}}\left\|F'(t,x,y)-F'_{R'}(t,x,y)\right\|_{H_1}
\left(\mu_{t+t_n}^x(dy)-\widetilde{\mu}_t^x(dy)\right)\nonumber\\
&
\leq \frac{C(1+\|x\|_{H_1}^2)}{R'}.
\end{align}
Combining \eqref{0220:02}, \eqref{0220:03}, \eqref{0220:04} and  \eqref{0220:05},
by \eqref{0220:01} and \eqref{0220:06},
we obtain that for any $0<R<R'$,
\begin{align*}
&
\lim_{n\rightarrow\infty}\sup_{|t|\leq l, \,\|x\|_{H_1}\leq R}
\left\|\bar{F}(t+t_n,x)-\int_{H_2}F'(t,x,y)\widetilde{\mu}_t^x(dy)\right\|\\
&
\leq\lim_{n\rightarrow\infty}\sup_{|t|\leq l,(x,y)\in \textbf{B}_{R'}}\|F(t+t_n,x,y)-F'(t,x,y)\|_{H_1}\\
&\quad
+\lim_{n\rightarrow\infty}{\color{blue}C_{R',l}}
\sup_{|t|\leq l,\|x\|_{H_1}\leq R}
d_{BL}(\mu_{t+t_n}^x,\widetilde{\mu}_t^x)
+\frac{C(1+R^2)}{R'}\\
&
\leq \frac{C(1+R^2)}{R'},
\end{align*}
which along with the arbitrary of $R'$ imply that $\{t_n\}\in\mathfrak{M}_{\bar{F}}$. Therefor, $\mathfrak{M}_{(F,\mu)}\subset\mathfrak{M}_{\bar{F}}$. The proof is complete.
\end{proof}

\begin{remark}\label{R2.5}
Since $\bar{F}(\cdot,\cdot)$ is uniformly almost periodic,
it follows from \cite[Theorem I.3.2]{Besi55} or \cite[Theorem 3.4]{CL2017} that
there exists $\bar{F}:H_1\rightarrow H_1$ such that for any $T>0$, $R>0$ and $x\in H_1$ with $\|x\|_{H_1}\leq R$,
\begin{align}
\sup_{t\in\RR}\left\|\frac{1}{T}\int^{T+t}_t \bar F(s,x)\,ds-\bar{F}(x)\right\|_{H_1}\leq \phi^R_2(T),\label{AsyF}
\end{align}
where $\phi^R_2(\cdot)$ satisfies $\phi^R_2(T)\rightarrow0$, as $T\rightarrow +\infty$.
\end{remark}

Note that using \eref{AsyF} and  \ref{C1}, it is easy to check $\bar{A}$ satisfies \ref{A1}-\ref{A4}, $\bar{F}$ and $\bar{G}_1$ are global Lipschitz continuous; see e.g. \cite[Lemma 4.2]{CL2023}.
Consequently, averaged equation \eref{AVE2} has a unique variational solution in the sense of Definition \ref{S.S.}, i.e., for its $dt\otimes \PP$-equivalence class $\check{\bar{X}}$ we have $\check{\bar{X}}\in L^{\alpha}([0, T]\times \Omega, dt\otimes \PP; V_1)\cap L^2([0, T]\times\Omega, dt\otimes \PP; H_1)$ with $\alpha$ as in \ref{A3}, we have $\PP$-a.s.
\begin{align}
\bar{X}_{t}=x+\int_{0}^{t}\bar A(\tilde{\bar{X}}_{s})ds+\int_{0}^{t}
\bar{F}(\bar{X}_{s})ds+\int_{0}^{t}\bar G_1(\bar{X}_{s})dW^{1}_s,\label{4.6b}
\end{align}
where $\tilde{\bar{X}}$ is any $V_1$-valued progressively measurable $dt\otimes \PP$-version of $\check{\bar{X}}$. Moreover, we also have the following estimates. Because their proofs follows almost the same steps in the proof of Lemmas \ref{PMY} and \ref{COX}, we omit them here.
\begin{lemma}\label{L4.5} For any $T>0$, $p\geq 1$, there exist constants $C_{p,T},C_T>0$ and $m>0$ such that for any $x\in H_1$,
\begin{align*}
\mathbb{E}\left(\sup_{t\in[0,T]}\|\bar{X}_{t}\|^{2p}_{H_1}\right)+\EE\left(\int_0^T\|\bar{X}_{t}\|^{2p-2}_{H_1}\|\tilde{\bar{X}}_{t}\|_{V_1}^{\alpha}dt\right)\leq C_{p,T}(1+\|x\|^{2p}_{H_1})
\end{align*}
and
\begin{align}
\mathbb{E}\left[\int^{T}_0\|\bar{X}_{t}-\bar{X}_{t(\delta)}\|^2_{H_1} dt\right]\leq C_{T}\delta^{1/2}(1+\|x\|^m_{H_1}).\label{FbarXT}
\end{align}
\end{lemma}

\vspace{0.1cm}

With the help of the above results,
the proof of Theorem \ref{main result 2} can be obtained by
modifying the methodology developed for \cite[Theorem 4.5]{CL2023}.
For the convenience of readers, we present an abbreviated version of the proof below.

\noindent\textbf{Proof of Theorem \ref{main result 2}:}
Applying It\^{o}'s formula, we have
\begin{align*}
\|\bar X_{t}^{\vare}-\bar{X}_{t}\|^2_{H_1}
=
&
2\int_0^t{_{V^{*}_1}}\langle A(s/\vare,\tilde{\bar X}_{s}^{\vare})-\bar{A}(\tilde{\bar X}_{s}),
\tilde{\bar X}_{s}^{\vare}-\tilde{\bar{X}}_{s}\rangle_{V_1} ds
+\tilde{\mathcal{M}}^{\vare}(t)\\
&
+\int_0^t\left(2\left\langle\bar F(s/\vare,\bar X_{s}^\varepsilon)-\bar{F}(\bar{X}_s), \bar{ X}_{s}^{\vare}-\bar{X}_{s}\right\rangle_{H_1}
+\|G_{1}(s/\vare,\bar{X}^{\vare}_s)-\bar G_1(\bar{X}_s)
\|_{L_{2}(U_1, H_1)}^2\right)ds,
\end{align*}
where
$$
\tilde{\mathcal{M}}^{\vare}(t):=2\int_0^t\langle \bar X_{s}^{\vare}-\bar{X}_{s}, [G_{1}(s/\vare,\bar{X}_s^\varepsilon)-\bar G_1(\bar{X}_s)]dW_s^{1}\rangle_{H_1}.
$$
Hence, by \ref{A2} and \ref{A5}, we have for any $\delta\in(0,1]$,
\begin{align*}
\|\bar X_{t}^{\vare}-\bar{X}_{t}\|^2_{H_1}
\leq
&
C\int_0^t\|\bar X_{s}^{\vare}-\bar{X}_{s}\|^2_{H_1}\left[\rho(\tilde{\bar{X}}^{\varepsilon}_s)
+\eta(\tilde{\bar{X}}_s)\right]ds\nonumber\\
&
+3\int_0^t\|G_{1}(s/\vare,\bar{X}_s^\varepsilon)-G_{1}(s/\vare,\bar{X}_s)
\|_{L_{2}(U_1, H_1)}^2ds\nonumber\\
&
+3\int_0^t\|G_{1}(s/\vare,\bar{X}_s)-G_{1}(s/\vare,\bar{X}_{s(\delta)})+\bar G_1(\bar{X}_{s(\delta)})-\bar G_1(\bar{X}_s)\|_{L_{2}(U_1, H_1)}^2ds\nonumber\\
&
+2\int_0^t\left\langle\left[\bar F(s/\vare,\bar X_{s(\delta)}^\varepsilon)-\bar{F}(\bar X^\varepsilon_{s(\delta)})\right], \bar X_{s}^{\vare}-\bar X_{s(\delta)}^{\vare}-\bar{X}_{s}+\bar{X}_{s(\delta)}\right\rangle_{H_1} ds\nonumber\\
&
+2\int_0^t\left\langle\left[\bar F(s/\vare,\bar X_{s}^\varepsilon)-\bar{F}(\bar X^{\vare}_s)-\bar F(s/\vare,\bar X_{s(\delta)}^\varepsilon)
+\bar{F}(\bar X^\varepsilon_{s(\delta)})\right], \bar X_{s}^{\vare}-\bar{X}_{s}\right\rangle_{H_1} ds\nonumber\\
&
+2\int_0^t\left\langle\left[\bar{F}(\bar X_{s}^\varepsilon)-\bar{F}(\bar{X}_s)\right], \bar X_{s}^{\vare}-\bar{X}_{s}\right\rangle_{H_1} ds
+\tilde{\sJ}^{\vare}_1(t)+2\tilde{\sJ}^{\vare}_1(t)+3\tilde{\sJ}^{\vare}_2(t)+\tilde{\mathcal{M}}^{\vare}(t)\\
\leq
&
C\int_0^t\|\bar X_{s}^{\vare}-\bar{X}_{s}\|^2_{H_1}\left[\rho(\tilde{\bar{X}}^{\varepsilon}_s)+\eta(\tilde{\bar{X}}_s)\right]ds\nonumber\\
&
+C\left[\int_0^t \left(1+\|\bar X_{s(\delta)}^\varepsilon\|_{H_1}^2\right)ds\right]^{1/2}
\left[\int^t_0 \left(\|\bar X^{\varepsilon}_{s}-\bar X^{\varepsilon}_{s(\delta)}\|^2_{H_1}+\|\bar X_{s}-\bar X_{s(\delta)}\|^2_{H_1} \right)ds\right]^{1/2}\nonumber\\
&
+C\int_0^t\!\left[\|X_{s}^{\vare}-X_{s(\delta)}^{\vare}\|^2_{H_1}\!+\!\|\bar{X}_s-\bar X_{s(\delta)}\|^2_{H_1}\right]ds
\!+\!\tilde{\sJ}^{\vare}_1(t)\!+\!2\tilde{\sJ}^{\vare}_1(t)\!+\!3\tilde{\sJ}^{\vare}_2(t)+\!\tilde{\mathcal{M}}^{\vare}(t),\nonumber\\
\end{align*}
where
\begin{align}\label{J1-J3}\begin{split}
&\tilde\sJ^{\vare}_1(t):=\int_{0}^{t}\left\|A(s/\vare,\tilde{\bar{X}}_s)-\bar{A}(\tilde{\bar{X}}_s)\right\|_{V^*_1}\|\tilde{\bar X}_{s}^{\vare}-\tilde{\bar{X}}_{s}\|_{V_1} ds,\\
&\tilde\sJ^{\vare}_2(t):=\int_0^t\left\langle\left[\bar F(s/\vare,\bar X_{s(\delta)}^\varepsilon)-\bar{F}(\bar X^\varepsilon_{s(\delta)})\right], \bar X_{s(\delta)}^{\vare}-\bar{X}_{s(\delta)}\right\rangle_{H_1} ds,\\
&\tilde\sJ^{\vare}_3(t):=\int_0^t\|G_{1}(s/\vare,\bar{X}_{s(\delta)})-\bar G_1(\bar{X}_{s(\delta)})
\|_{L_{2}(U_1, H_1)}^2ds.
\end{split}
\end{align}
Define $\bar{\tau}^{\vare}_R:=\bar{\tau}^{1,\vare}_R\wedge\bar{\tau}^2_R$ with $\bar{\tau}^{1,\vare}_R$ and $\bar{\tau}^2_R$ are defined as follows:
\begin{align*}
&\bar{\tau}^{1,\vare}_R:=\inf\left\{t\geq 0:\int_0^t(1+\|\tilde{\bar{X}}^{\varepsilon}_s\|_{V_1}^{\alpha})(1+\|\bar{X}^{\varepsilon}_s\|_{H_1}^{\beta})ds\geq R~\text{or}~\|\bar{X}^{\varepsilon}_t\|_{H_1}\geq R\right\},\\
&\bar{\tau}^2_R:=\inf\left\{t\geq 0:\int_0^t(1+\|\tilde{\bar{X}}_s\|_{V_1}^{\alpha})(1+\|\bar{X}_s\|_{H_1}^{\beta})ds\geq R ~\text{or}~\|\bar{X}_t\|_{H_1}\geq R\right\}.
\end{align*}
Then by the same argument as used in the proof \textbf{Step 1} of Proposition \ref{ESX}, we can easily obtain
\begin{align}\label{0528:10}
\mathbb{E}\left(\sup_{t\in[0, T\wedge\bar{\tau}^{\vare}_R]}
\|\bar X_{t}^{\vare}-\bar{X}_{t}\|^2_{H_1}\right)
\leq&
C_{T}(1+\|x\|^m_{H_1}+\|y\|^m_{H_2})\delta^{1/4}\nonumber\\
&+C_{R,T}\sum_{i=1}^3\EE\left(\sup_{t\in[0, T\wedge\bar{\tau}^{\vare}_R]}
\left|\tilde{\sJ}^{\vare}_i(t)\right|\right).
\end{align}

Note that by H\"older's inequality and the definition of $\bar{\tau}_R^\varepsilon$, we have
\begin{align}\label{0721:01}
&
\EE\left(\sup_{t\in[0, T\wedge\bar{\tau}^{\vare}_R]}
\left|\tilde\sJ^{\vare}_1(t)\right|\right)\nonumber\\
&
=\EE\left(\sup_{t\in[0, T\wedge\bar{\tau}^{\vare}_R]}
\int_{0}^{t}\left\|A(s/\vare,\tilde{\bar{X}}_s)
-\bar{A}(\tilde{\bar{X}}_s)\right\|_{V^*_1}
\|\tilde{\bar X}_{s}^{\vare}-\tilde{\bar{X}}_{s}\|_{V_1} ds\right)\nonumber\\
&
\leq C\EE\left(\sup_{t\in[0, T\wedge\bar{\tau}^{\vare}_R]}
\left(\int_0^t\left\|A(s/\vare,\tilde{\bar{X}}_s)
-\bar{A}(\tilde{\bar{X}}_s)
\right\|_{V^*_1}^{\frac{\alpha}{\alpha-1}}ds
\right)^{\frac{\alpha-1}{\alpha}}
\left(\int_0^t\|\tilde{\bar X}_{s}^{\vare}\|_{V_1}^\alpha
+\|\tilde{\bar{X}}_{s}\|_{V_1}^\alpha ds \right)^{\frac{1}{\alpha}}
\right)\nonumber\\
&
\leq C_R\left(\EE
\int_0^T\left\|A(s/\vare,\tilde{\bar{X}}_s)
-\bar{A}(\tilde{\bar{X}}_s)
\right\|_{V^*_1}^{\frac{\alpha}{\alpha-1}}ds
\right)^{\frac{\alpha-1}{\alpha}}.
\end{align}
And it follows from \ref{A4}, Remark \ref{Rem0722},
\eqref{F3.1} and Lemma \ref{L4.5} that
\begin{align*}
\EE
\int_0^T\left\|A(s/\vare,\tilde{\bar{X}}_s)
-\bar{A}(\tilde{\bar{X}}_s)
\right\|_{V^*_1}^{\frac{\alpha}{\alpha-1}}ds
&
\leq C\EE\int_0^T\left(1+\|\tilde{\bar{X}}_s\|_{V_1}^\alpha\right)
\left(1+\|\bar{X}_s\|_{H_1}^\beta\right)ds\\
&
\leq C\left(1+\|x\|_{H_1}^{\beta+2}+\|y\|_{H_2}^{\beta+2}
\right),
\end{align*}
which along with \eqref{0721:01} and Lebesgue dominated convergence theorem implies that

\begin{align}
\lim_{\vare\rightarrow 0}\EE\left(\sup_{t\in[0, T\wedge\bar{\tau}^{\vare}_R]}
\left|\tilde\sJ^{\vare}_1(t)\right|\right)=0.\label{tildeJ1}
\end{align}

Note that
\begin{align}
\tilde \sJ^{\vare}_2(t)
\leq&
\sum_{k=0}^{[t/\delta]-1}
\left|\int_{k\delta}^{(k+1)\delta}\langle \bar{F}(s/\vare,\bar X_{k\delta}^{\vare})-\bar{F}(\bar X^{\vare}_{k\delta}), \bar X_{k\delta}^{\vare}-\bar{X}_{k\delta}\rangle_{H_1} ds\right|\nonumber\\
&
+\left|\int_{t(\delta)}^{t}\langle \bar{F}(s/\vare,\bar X_{s(\delta)}^{\vare})
-\bar{F}(\bar X^{\vare}_{s(\delta)}), \bar X_{s(\delta)}^{\vare}-\bar{X}_{s(\delta)}\rangle_{H_1} ds\right|
=:\tilde \sJ^{\vare}_{21}(t)+\tilde \sJ^{\vare}_{22}(t). \label{S}
\end{align}
For the term $\tilde \sJ^{\vare}_{22}(t)$, by H\"older's inequality, \eqref{F3.1} and Lemma \ref{L4.5},
it is easy to see
\begin{align}
\EE\left[\sup_{t\in [0, T\wedge{\bar{\tau}^{\vare}_R}]}\tilde \sJ^{\vare}_{22}(t)\right]
\leq&
C\left[\EE\sup_{t\in [0, T\wedge{\bar{\tau}^{\vare}_R}]}\left(\|\bar X^{\vare}_t\|_{H_1}^2+\|\bar{X}_{t}\|^2_{H_1}\right)\right]^{1/2}\left[\EE\int_{0}^{T\wedge{\bar{\tau}^{\vare}_R}}(1+\|\bar X^{\vare}_{s(\delta)}\|^2_{H_1})ds\right]^{1/2}\delta^{1/2}\nonumber\\
\leq&
C_{R,T}\delta^{1/2}.\label{S2}
\end{align}

For the term $\tilde \sJ^{\vare}_{21}(t)$, we have
\begin{align*}
&
\mathbb{E}\left[\sup_{t\in[0, T\wedge{\bar{\tau}^{\vare}_R}]}\tilde \sJ^{\vare}_{21}(t)\right]\\
\leq&
\mathbb{E}\left(\sum_{k=0}^{[(T\wedge{\bar{\tau}^{\vare}_R})/\delta]-1}
\left|\int_{k\delta}^{(k+1)\delta}\langle \bar{F}(s/\vare,\bar X_{k\delta}^{\vare})-\bar{F}(\bar X_{k\delta}^{\vare}), \bar X_{k\delta}^{\vare}-\bar{X}_{k\delta}\rangle_{H_1} ds\right|\right)\nonumber\\
\leq&
\frac{C_{T}}{\delta}\max_{0\leq k\leq[T/\delta]-1}\mathbb{E}\left[\left|\int_{k\delta}^{(k+1)\delta}
\langle \bar{F}(s/\vare,\bar X_{k\delta}^{\vare})-\bar{F}(\bar X_{k\delta}^{\vare}), \bar X_{k\delta}^{\vare}-\bar{X}_{k\delta}\rangle_{H_1} ds\right|\cdot 1_{\{k\delta<{\bar{\tau}^{\vare}_R}\}}\right]\nonumber\\
\leq&
\frac{C_{T}}{\delta}\!\max_{0\leq k\leq[T/\delta]-1}\!\mathbb{E}\left[\|\bar X^{\vare}_{k\delta}-\bar X_{k\delta}\|_{H_1}\left\|\int_{k\delta}^{(k+1)\delta}
\bar{F}(s/\vare,\bar X_{k\delta}^{\vare})-\bar{F}(\bar X_{k\delta}^{\vare})ds\right\|_{H_1}\cdot 1_{\{k\delta<{\bar{\tau}^{\vare}_R}\}}\right]\nonumber\\
=&
C_T\!\max_{0\leq k\leq[T/\delta]-1}\!\mathbb{E}\left[\|\bar X^{\vare}_{k\delta}-\bar X_{k\delta}\|_{H_1}\left\|\frac{\vare}{\delta}\int_{\frac{k\delta}{\vare}}^{\frac{k\delta}{\vare}+\frac{\delta}{\vare}}
\bar F(s,\bar X_{k\delta}^{\vare})ds-\bar{F}(\bar X_{k\delta}^{\vare})\right\|_{H_1}\cdot 1_{\{k\delta<{\bar{\tau}^{\vare}_R}\}}\right].\nonumber
\end{align*}

By \eref{AsyF} in Remark \ref{R2.5}, we have for any $0\leq k\leq[T/\delta]-1$,
\begin{align*}
\left\|\frac{\vare}{\delta}\int_{\frac{k\delta}{\vare}}^{\frac{k\delta}{\vare}+\frac{\delta}{\vare}}
\bar F(s,\bar X_{k\delta}^{\vare})ds-\bar{F}(\bar X_{k\delta}^{\vare})\right\|_{H_1}\cdot 1_{\{k\delta<{\bar{\tau}^{\vare}_R}\}}
\leq \phi^{R}_{2}(\delta/\vare).
\end{align*}
Hence, we get
\begin{align}\label{0528:07}
\mathbb{E}\left(\sup_{t\in[0, T\wedge{\bar{\tau}^{\vare}_R}]}\tilde \sJ^{\vare}_{21}(t)\right)
\leq C_{R,T}\phi^{R}_{2}(\delta/\vare).
\end{align}
Combining \eqref{S}, \eqref{S2} and \eqref{0528:07}, we have
\begin{align}\label{0528:08}
\EE\left(\sup_{t\in[0, T\wedge{\bar{\tau}^{\vare}_R}]}
\left|\tilde\sJ^{\vare}_2(t)\right|\right)
\leq C_{R,T}\left(\phi^{R}_{2}(\delta/\vare)+\delta^{1/2}\right).
\end{align}
By a similar argument above and condition \eref{C1eq}, we still have
\begin{align}\label{0528:09}
\EE\left(\sup_{t\in[0, T\wedge{\bar{\tau}^{\vare}_R}]}
\left|\tilde\sJ^{\vare}_3(t)\right|\right)\leq C_{R,T}\left(\phi^R_1(\delta/\vare)
{+\delta}\right).
\end{align}

Finally, taking $\delta=\vare^{2/3}$, by \eqref{0528:10}, \eqref{0528:08}
and \eqref{0528:09}, we obtain
\begin{align*}
\mathbb{E}\left(\sup_{t\in[0, T\wedge{\bar{\tau}^{\vare}_R}]}\|\bar X_{t}^{\vare}-\bar{X}_{t}\|^2_{H_1}\right)
\leq
C_{R,T}\left[\phi^R_1(\vare^{-1/3})
+\phi^R_2(\vare^{-1/3})+\vare^{1/3}\right]
+C_{R,T}\EE\left(\sup_{t\in[0, T\wedge\bar{\tau}^{\vare}_R]}
\left|\tilde\sJ^{\vare}_1(t)\right|\right),
\end{align*}
which implies
\begin{align*}
\mathbb{E}\left(\sup_{t\in [0, T]}\|\bar X_{t}^{\vare}-\bar{X}_{t}\|^2_{H_1}\right)\leq
&
\mathbb{E}\left(\sup_{t\in [0, T]}\|\bar X^{\vare}_{t}-\bar{X}_{t}\|^2_{H_1} 1_{\{T\leq \bar \tau^{\vare}_{R}\}}\right)+\mathbb{E}\left(\sup_{t\in [0, T]}\|\bar X_{t}^{\vare}-\bar{X}_{t}\|^2_{H_1} 1_{\{T>\bar \tau^{\vare}_{R}\}}\right)\nonumber\\
\leq
&
C_{R,T}\! \left[\phi^R_1(\vare^{-1/3})+\phi^R_2(\vare^{-1/3})+\vare^{\frac{1}{3}}\right]
\!+\!C_{R,T}\EE\left(\sup_{t\in[0, T\wedge\bar{\tau}^{\vare}_R]}
\left|\tilde\sJ^{\vare}_1(t)\right|\right)\\
&
+\frac{C_{T}{\left(1+\|x\|^m_{H_1}+\|y\|^m_{H_2}\right)}}
{\sqrt{R}}.
\end{align*}
Now using \eqref{tildeJ1} and letting $\vare\rightarrow 0$ firstly, then $R\rightarrow \infty$, it follows
\begin{eqnarray*}
\lim_{\vare\rightarrow 0}\mathbb{E}\left(\sup_{t\in [0, T]}\|\bar X_{t}^{\vare}-\bar{X}_{t}\|^2_{H_1}\right)=0.
\end{eqnarray*}
Hence, using Lemmas \ref{L3.8} and \ref{L4.5}, we finally get for any $p\geq 1$,
$$
\lim_{\vare\rightarrow 0}\EE\left(\sup_{t\in [0, T]}\|\bar X_{t}^{\vare}-\bar{X}_{t}\|^{2p}_{H_1}\right)=0,
$$
which together with \eref{2.2} imply \eref{2.3} holds. The proof is complete.

\vspace{0.3cm}
\textbf{Acknowledgment}.
Mengyu Cheng is supported by the NSF of China (No. 12301223)
and the Fundamental Research Funds for the Central Universities.
Xiaobin Sun is supported by the NSF of China (Nos.
12271219, 12090010 and 12090011). Yingchao Xie is supported by the NSF of China  (No 12471139) and the Priority Academic Program Development of Jiangsu Higher Education Institutions.

\vspace{0.3cm}
\textbf{Data availability}

\vspace{0.2cm}
No data was used for the research described in the article.

%

\vspace{0.3cm}

\end{document}